\documentclass[12pt]{amsart}
\usepackage{amsmath, amssymb, amsthm, amsfonts, enumerate, verbatim}
\usepackage[all]{xy}
\usepackage{tikz}
\tikzstyle{punkt}=[circle, fill=black, minimum size=1mm,inner sep=0pt, draw]
\def\NZQ{\mathbb}               

\def\ZZ{{\NZQ Z}}
\def\RR{{\NZQ R}}

\def\FF{{\NZQ F}}

\def\frk{\frak}               

\def\Phi{{\frk n}}
\def\Phi{{\frk N}}
\def\MC{{\mathcal C}}

\def\KK{{\mathbb K}}
\def\opn#1#2{\def#1{\operatorname{#2}}} 
\opn\chara{char}
\opn\length{\ell}
\opn\pd{pd}
\opn\rk{rk}
\opn\projdim{proj\,dim}
\opn\injdim{inj\,dim}
\opn\rank{rank}
\opn\depth{depth}
\opn\grade{grade}
\opn\height{height}
\opn\embdim{emb\,dim}
\opn\codim{codim}

\opn\Tr{Tr}
\opn\bigrank{big\,rank}
\opn\superheight{superheight}
\opn\lcm{lcm}
\opn\trdeg{tr\,deg}
\opn\reg{reg}
\opn\lreg{lreg}
\opn\ini{in}
\opn\lpd{lpd}
\opn\size{size}
\opn\bigsize{bigsize}
\opn\cosize{cosize}
\opn\bigcosize{bigcosize}
\opn\sdepth{sdepth}
\opn\sreg{sreg}
\opn\link{link}
\opn\fdepth{fdepth}
\opn\lin{lin}
\opn\ini{in}
\opn\div{div}
\opn\Div{Div}
\opn\cl{cl}
\opn\Cl{Cl}
\opn\Spec{Spec}
\opn\Supp{Supp}
\opn\supp{supp}
\opn\Sing{Sing}
\opn\Ass{Ass}
\opn\Min{Min}
\opn\Mon{Mon}
\opn\dstab{dstab}
\opn\astab{astab}
\opn\Syz{Syz}
\opn\Ann{Ann}
\opn\Rad{Rad}
\opn\Soc{Soc}
\opn\Im{Im}
\opn\Ker{Ker}
\opn\Coker{Coker}
\opn\Am{Am}
\opn\Hom{Hom}
\opn\Tor{Tor}
\opn\Ext{Ext}
\opn\End{End}
\opn\Aut{Aut}
\opn\id{id}

\opn\nat{nat}
\opn\pff{pf}
\opn\Pf{Pf}
\opn\GL{GL}
\opn\SL{SL}
\opn\mod{mod}
\opn\ord{ord}
\opn\Gin{Gin}
\opn\Hilb{Hilb}
\opn\sort{sort}
\opn\initial{init}
\opn\ende{end}
\opn\height{height}
\opn\type{type}
\opn\mdeg{mdeg}
\opn\aff{aff}
\opn\con{conv}
\opn\relint{relint}
\opn\st{st}
\opn\lk{lk}
\opn\cn{cn}
\opn\core{core}
\opn\vol{vol}
\opn\link{link}
\opn\star{star}
\opn\lex{lex}
\opn\sign{sign}
\opn\gr{gr}

\def\pot#1#2{#1[\kern-0.28ex[#2]\kern-0.28ex]}

\opn\dirlim{\underrightarrow{\lim}}
\opn\inivlim{\underleftarrow{\lim}}
\let\iso=\cong

\def\Implies{\ifmmode\Longrightarrow \else
        \unskip${}\Longrightarrow{}$\ignorespaces\fi}
\def\implies{\ifmmode\Rightarrow \else
        \unskip${}\Rightarrow{}$\ignorespaces\fi}
\def\iff{\ifmmode\Longleftrightarrow \else
        \unskip${}\Longleftrightarrow{}$\ignorespaces\fi}

\let\:=\colon
\newtheorem{Theorem}{Theorem}[section]
 \newtheorem{Lemma}[Theorem]{Lemma}
 \newtheorem{Corollary}[Theorem]{Corollary}
 \newtheorem{Proposition}[Theorem]{Proposition}
 \newtheorem{Remark}[Theorem]{Remark}
 
 \newtheorem{Example}[Theorem]{Example}
 
 \newtheorem{Definition}[Theorem]{Definition}
 \newtheorem{Problem}[Theorem]{Problem}
 \newtheorem{Conjecture}[Theorem]{Conjecture}
 
\let\epsilon\varepsilon
\let\kappa=\varkappa
\def\pnt{{\raise0.5mm\hbox{\large\bf.}}}

\begin{document}
\title{Cycle algebras and polytopes of matroids}
\author {Tim R\"omer and Sara Saeedi Madani}

\address{Tim R\"omer, Universit\"at Osnabr\"uck, Institut f\"ur Mathematik, 49069 Osnabr\"uck, Germany}
\email{troemer@uni-osnabrueck.de}

\address{Sara Saeedi Madani, Faculty of Mathematics and Computer Science, Amirkabir University of Technology (Tehran Polytechnic), Tehran, Iran, and School of Mathematics, Institute for Research in Fundamental Sciences (IPM), P.O. Box: 19395‐
	5746, Tehran, Iran}
\email{sarasaeedi@aut.ac.ir}

\begin{abstract}
Cycle polytopes of matroids have been introduced in combinatorial optimization as a generalization of important classes of polyhedral objects like cut polytopes and Eulerian subgraph polytopes associated to graphs. Here we start an algebraic and geometric investigation of these polytopes by studying their toric algebras, called cycle algebras, and their defining ideals. Several matroid operations are considered which determine faces of cycle polytopes that belong again to this class of polyhedral objects. As a key technique used in this paper, we study certain minors of given matroids which yield algebra retracts on the level of cycle algebras. In particular, that allows us to use a powerful algebraic machinery. As an application, we study highest possible degrees in minimal homogeneous systems of generators of defining ideals of cycle algebras as well as interesting cases of cut polytopes and Eulerian subgraph polytopes.		
\end{abstract}

\thanks{
	 The research of the second author was in part supported by a grant from IPM (No. 1400130113).
}


\maketitle

\section{Introduction}\label{introduction}

There are various ways to associate polyhedra to objects of interest in combinatorial optimization and discrete mathematics. A prominent example with many applications is the \emph{cut polytope}~${\mathrm{Cut}}^{\square}(G)$ of a graph~$G$ which has taken the attention of researchers from many fields. Cut polytopes are well studied, see for example~\cite{BM, D, DD, DL1}. Indeed, cut polytopes are closely related to some well-known problems, like the \emph{MAXCUT} problem (see for example~\cite{DL, DL2, DL3}) and the \emph{four color} theorem in graph theory (see \cite{LM}).
Several geometric properties of cut polytopes of graphs have been studied for instance in~\cite{CKNR, KR, LM, O1}.  For more information about cut polytopes of graphs, see in particular the book of Deza and Laurent~\cite{DL}. 

An important class of classical combinatorial objects is the one of matroids. As a generalization of polyhedral objects like cut polytopes, Barahona and Gr\"otschel introduced in \cite{BGr} 
the cycle polytope $P_{\mathrm{Cyc}}(M)$ associated to a matroid $M$. 
From the geometric point of view these polytopes are the core objects studied in this paper. Observe that in the special case that the underlying matroid is the cographic matroid~$M(G)^*$ of a graph~$G$, the cycle polytope coincides with the cut polytope~${\mathrm{Cut}}^{\square}(G)$. A description of the facets of cycle polytopes of binary matroids as well as other properties of interest can be found for instance in \cite{BS, BGr, CP, GLPT, GT, GT1, KS}.  

Another special case of interesting cycle polytopes of matroids arises when the underlying matroid $M(G)$ is graphic for a given graph $G$. This polytope is called the \emph{Eulerian subgraph polytope} of~$G$ and we denote it by~$\mathrm{Euler}(G)$, see for example~\cite{BGr}. Eulerian subgraphs occur in various context in graph theory, see for example~\cite{Catlin}.  

In general, there are also toric algebras attached to $0/1$-polytopes which have been of great interest from the point of view of both algebraic geometry and commutative algebra. See the books~\cite{BG, CLS, MS} concerning such algebras and their geometrical aspects. In the particular case of cut polytopes, the aforementioned toric algebras and their defining ideals, which were studied first in~\cite{SS}, are called \emph{cut algebras} and \emph{cut ideals}, respectively. For further studies around cut algebras and ideals, see, e.g.,~\cite{En, KR, KNP, NP, O2, PS, RS}. For applications to algebraic statistics related to binary graph models, Markov random fields and phylogenetic models on split systems as a generalization of binary Jukes-Cantor models, see for example~\cite{SS}.     

In this paper, besides a better understanding of cycle polytopes of matroids in general, we study the associated toric algebras and ideals, called \emph{cycle algebras} and \emph{cycle ideals}. 
One of our main approaches is based on the investigation of certain operations on matroids to obtain faces of cycle polytopes,
which belong again to this class of objects, as well as induced algebra retracts of cycle algebras.

The organization of the paper is as follows. In Section~\ref{matroids}, we recall some ingredients on matroid theory used throughout the paper. In particular, we give a brief overview on some classical operations on matroids as well as  well-known classes of matroids. 

In Section~\ref{cycle polytopes}, cycle polytopes are defined and basic properties of them are studied. In particular,
operations on matroids are discussed which
yield faces of cycle polytopes belonging again to this class  of polytopes. We also pay particular attention to two special classes of cycle polytopes: cut and Eulerian subgraph polytopes.

In Section~\ref{algebra retracts}, we introduce cycle algebras and cycle ideals of matroids. Retracts of cycle algebras are considered to understand transition phenomena of algebraic properties of interest.
In particular, we study retracts obtained by faces of cycle polytopes and very useful retracts which do not arise in this way. 
For this purpose, we first define the new notion of a \emph{matroidal retract} of a matroid in Definition~\ref{matroidal retract-def}. 
Then, as one of our main results, Theorem~\ref{matroidal retract-theorem} states that matroidal retracts induce algebra retracts. 
We also introduce \emph{binary matroidal retracts} as a special case of matroidal retracts (see Definition~\ref{binary matroidal retract-def}). 
Finally, combining binary matroidal retracts with classical deletions as well as certain types of contractions, we get a new type of minors of binary matroids, which we call \emph{generalized series minors} (in short, g-series minors) of a given binary matroid and which are crucial in the rest of the paper. 

In Section~\ref{Cographic case}, we study g-series minors in the case of cographic matroids. Theorem~\ref{neighborhood-g-series minor-corollary} has in particular as a corollary the main results in \cite{RS}, which provides algebra retracts of cut algebras induced by neighborhood-minors of graphs.        

In Section~\ref{highest degree}, we study highest possible degrees in minimal homogeneous systems of generators of cycle ideals. As starting point, in Lemma~\ref{zero ideal} it is characterized when the ideal is zero in terms of data of the underlying matroid. 
In Lemma~\ref{mu} it is also observed that cycle ideals never contain linear forms. 
Moreover, Corollary~\ref{mu comparison} yields 
inequalities between the highest degree $\mu(M)$ of minimal homogeneous systems of generators of cycle ideals of a matroid $M$ and the ones obtained by various types of minors of $M$. In Theorem~\ref{simplification} we get cases where such inequalities are equalities.

In Section~\ref{degree 2}, we focus on small values of $\mu(M)$ for binary matroids. We discuss certain necessary and sufficient conditions for $\mu(M)\leq 2$ and $\mu(M)\leq 5$, respectively. The aforementioned conditions are given in terms of different excluded minors. In particular, all graphic and cographic matroids~$M$ with $\mu(M)\leq 2$ are classified in one of the main results of the paper, Theorem~\ref{degree2-characterization}. We also discuss 
the relationship of the results of this section and two conjectures posed in~\cite{SS}.   

Throughout the paper, we discuss examples and also pose several problems and conjectures arising from different parts of our work.

\section{Ingredients from matroid theory}\label{matroids}

In this section, we give a brief overview of some properties of  matroids and certain operations on them as well as of some well-known classes of matroids which are of importance for this work. For a general discussion on matroids see, e.g., \cite{Ox, Ox1}.   

Let $M$ be a matroid on the ground set $E(M)$, whose set of \emph{independent} sets, \emph{bases} (i.e.~maximal independent sets) and \emph{circuits} (i.e.~minimal dependent sets) are denoted by $\mathcal{I}(M)$, $\mathcal{B}(M)$ and $\mathcal{C}(M)$, respectively. We may also write $E$, $\mathcal{I}$, $\mathcal{B}$ and $\mathcal{C}$ for the aforementioned sets if the matroid $M$ is known from the context.  

The \emph{dual} matroid $M^*$ of $M$ is the matroid with the same ground set $E$ as $M$ and whose set of bases is defined as $\mathcal{B}(M^*)=\{E-B:B\in \mathcal{B}(M)\}$.      
It is well-known that a set $C\subseteq E$ is a circuit of $M^*$ if and only if it is a minimal set having non-empty intersection with every basis of $M$. The elements of $\mathcal{B}(M^*)$ and $\mathcal{C}(M^*)$ are also called \emph{cobases} and \emph{cocircuits} of $M$. For any circuit $C$ and cocircuit $C^*$, one has 
\begin{equation}\label{intersection of circuit and cocircit}
|C\cap C^*|\neq 1; 
\end{equation} 
see, e.g., \cite[Proposition~2.1.11]{Ox}. Observe that $(M^*)^*=M$.    

An element $e\in E$ is called a \emph{loop} of $M$ if $\{e\}$ is a circuit. A pair of elements $e,f\in E$ are called \emph{parallel} in $M$ if $\{e,f\}$ is a circuit. A \emph{parallel class} of $M$ is a maximal subset of $E$ with the property that any two distinct elements of it are parallel, and no element is a loop. The loops, parallel elements and parallel classes of the dual matroid $M^*$ are called \emph{coloops}, \emph{coparallel elements} and \emph{coparallel classes} of $M$. Coparallel classes of $M$ are also known as its \emph{series classes}. A matroid is said to be \emph{simple} (resp.~\emph{cosimple}) if it has no loops (resp.~coloops) and no non-trivial parallel (resp.~coparallel) classes. In particular, if $M$ is simple (resp.~cosimple), then $\{e\}$ is a parallel (resp.~coparallel) class of $M$ for any $e\in E$ and these are the only parallel (resp.~coparallel) classes.

Recall that $M$ is connected if and only if for any two distinct elements of $E$, there is a circuit containing both of them; see, e.g., \cite[Proposition~4.1.4]{Ox}. It is well-known that $M$ is connected if and only if $M^*$ is connected; see, e.g.,  \cite[Corollary~4.2.8]{Ox}.        

Next, we recall two important matroid operations, namely deletion and contraction. By the \emph{deletion} of $e\in E$ from $M$, one obtains a matroid denoted by $M\setminus e$ with the ground set $E-\{e\}$ and  
$
\mathcal{C}(M\setminus e)=\{C\subseteq E-\{e\}:C\in \mathcal{C}(M)\}. 
$
Repeating this procedure yields a deletion of a subset of the ground set. In particular, if $X\subseteq E$, then the \emph{restriction} of $M$ to $X$ is defined as the deletion of $E-X$ from $M$, which is denoted by $M|X$. This is the matroid on $X$ with 
$
\mathcal{C}(M|X)=\{C\subseteq X:C\in \mathcal{C}(M)\}.
$

If $e$ is not a loop of $M$, then by the \emph{contraction} of $e$ in $M$, one gets a matroid, denoted by $M/e$, with the ground set $E-\{e\}$ whose 
circuits are the minimal elements of 
$\{C-\{e\}: C \in \mathcal{C}(M)\}$. If $e$ is a loop of $M$, then by definition $M/e:=M\setminus e$. 
Similar to deletion, one can contract a subset $T$ of $E$ in $M$. Then one obtains the matroid $M/T$ on the ground set $E-T$ and 
whose circuits are the minimal non-empty elements of 
$\{C-T: C \in \mathcal{C}(M)\}$. 

Observe that duality, deletion and contraction are related to each other as follows:
\begin{equation}\label{dual-deletion-contraction}
M^*/T={(M\setminus T)}^*~~~~\text{and}~~~~M^*\setminus T={(M/T)}^*.
\end{equation}

A \emph{minor} of a matroid $M$ is a matroid which can be obtained from $M$ by a sequence of deletions and contractions. There are special types of minors in the literature such as \emph{parallel minors} and \emph{series minors} which we recall in the following: A \emph{parallel minor} of $M$ is a matroid which can be obtained from $M$ by a sequence of contractions and \emph{parallel deletions}. Here, a parallel deletion of $M$ is a matroid of the form $M\setminus e$ where $e$ is contained in a $2$-circuit of $M$. A \emph{series minor} of $M$ is a matroid which can be obtained from $M$ by a sequence of deletions and \emph{series contractions}. Here, a series contraction of $M$ is a matroid of the form $M/e$ where $e$ is contained in a $2$-cocircuit of $M$. 

Observe that $N$ is a parallel minor of $M$ if and only if $N^*$ is a series minor of $M^*$.

Let $M$ and $N$ be two matroids. Then $M$ is said to be $N$-\emph{minor~free} if $M$ has no minor isomorphic to $N$. The notion of $N$-\emph{parallel minor~free} and $N$-\emph{series minor~free} are defined analogously.  

Next, we recall another useful operation on matroids. Let $M_1$ and $M_2$ be two matroids with $E(M_1)\cap E(M_2)=\emptyset$. The \emph{direct sum} or \emph{$1$-sum} $M_1\oplus M_2$ of $M_1$ and $M_2$ is the matroid with $E(M_1\oplus M_2)=E(M_1)\cup E(M_2)$, 
\[
\mathcal{I}(M_1\oplus M_2)=\{I_1\cup I_2: I_i\in \mathcal{I}(M_i),i=1,2\}\quad \text{and} \quad 
\mathcal{C}(M_1\oplus M_2)=\mathcal{C}(M_1)\cup \mathcal{C}(M_2). 
\]
In the remaining part of this section, we briefly recall important classes of matroids which will be used throughout the paper. Let $m,n$ be non-negative integers. The \emph{uniform matroid} $U_{m,n}$ is a matroid on the ground set $E$ of cardinality $n$ whose independent sets are exactly the subsets of $E$ of cardinality at most $m$. Therefore,
\[
\mathcal{C}(U_{m,n})=\{C\subseteq E: |C|=m+1\}.
\]
In particular, the uniform matroids $U_{n,n}$ have no circuits. Indeed, they are the only matroids with this property, and are called \emph{free}. Moreover, $U_{0,0}$, which is the unique matroid with the empty ground set, is called the 
\emph{empty matroid}.    

Let $A$ be an $m\times n$-matrix over a field $\KK$, and let $E$ be the set of column labels of $A$. Moreover, let $\mathcal{I}$ be the set of subsets $X$ of $E$ for which the multiset of columns labeled by $X$ is linearly independent in $\KK^m$. Then it is easily seen that $(E,\mathcal{I})$ is a matroid which is called the \emph{vector matroid} of $A$. A matroid $M$ which is isomorphic to the vector matroid of a matrix  over a field $\KK$ is said to be \emph{representable} over $\KK$. 
A matroid which is representable over the field $\FF_2$, i.e.~the unique field with two elements, is called
\emph{binary}. There are various characterizations of binary matroids and here we summarize some of them which are used throughout the paper:

\begin{Theorem}\label{binary}
	{\em(}\cite[Theorem~9.1.2 and Theorem~9.1.5]{Ox}{\em)}
	Let $M$ be a matroid. Then the following statements are equivalent:
	\begin{enumerate}
		\item[{\em(a)}] $M$ is binary;
		\item[{\em(b)}] $M$ is $U_{2,4}$-minor-free;
		\item[{\em(c)}] For every circuit $C$ and cocircuit $C^*$ of $M$, $|C\cap C^*|$ is even;
		\item[{\em(d)}] If $C_1$ and $C_2$ are distinct circuits of $M$, then 
		$C_1\Delta C_2=(C_1\cup C_2)\setminus (C_1\cap C_2)$
		is a disjoint union of circuits. 
	\end{enumerate}
\end{Theorem}  

It is straightforward from Theorem~\ref{binary} that the dual of a binary matroid is binary as well. Moreover, the class of binary matroids is minor-closed; see, e.g., \cite[Proposition~3.2.4]{Ox}. In particular, all different types of minors of binary matroids, which were mentioned in this section, are again binary matroids. 

Let $G$ be a graph with the edge set $E=E(G)$. Attached to $G$ is the matroid $M(G)$ on the ground set $E$, whose circuits are exactly the edge sets of cycles of $G$. The matroid $M(G)$ is called the \emph{cycle matroid} or \emph{polygon matroid} of $G$. 
In particular, a loop and a parallel class in $M(G)$, respectively, correspond to a loop and a maximal set of pairwise parallel (or multiple) edges in $G$, respectively. Thus, $M(G)$ is a simple matroid if and only if $G$ is a simple graph, i.e.~a graph with no loops and no parallel edges. For a graph $G$ with at least three vertices and no isolated vertices and loops, one has that $M(G)$ is connected if and only if $G$ is a $2$-connected graph; see, e.g., \cite[Proposition~4.1.8]{Ox}.   

Let $e\in E$ which is not a loop of $G$. Then the graph $G\setminus e$ denotes the graph on the same vertex set as $G$ obtained by deleting the edge $e$ from $G$. Moreover, the graph $G/e$ denotes the graph which is obtained from $G$ by identifying the endpoints of $e$ and deleting $e$. This operation is called the contraction of the edge $e$ in $G$. Then 
\begin{equation}\label{graphic-deletion-contraction}
M(G)\setminus e=M(G\setminus e)~~~~~~~\text{and}~~~~~~~M(G)/e=M(G/e),
\end{equation}   
for any $e\in E$. Note that similar to matroids, if $e$ is a loop in $G$, then $G/e=G\setminus e$. Observe that even if $G$ is a simple graph, then $G/e$ has not to be simple.     

Any matroid isomorphic to $M(G)$ for some graph $G$, is called a \emph{graphic matroid}. It is well-known that graphic matroids are representable over any field, and in particular, they are binary matroids; see, e.g., \cite[Proposition~5.1.2]{Ox}. It follows immediately from the well-known Whitney's theorem that for any graphic matroid $M$, there exists a connected graph $G$ such that $M$ is isomorphic to $M(G)$. 

A \emph{cographic matroid} $M(G)^*$ is the dual of a graphic matroid $M(G)$. This is one way to see that cographic matroids are also binary. Recall that an \emph{edge cut} in a graph $G$ is a set of edges $X$ such that $G\setminus X$ has more connected components than $G$. If $X$ consists of only one edge $e$, then $e$ is called a \emph{bridge} of $G$. Observe that the circuits of $M(G)^*$ are exactly the minimal edge cuts of $G$; see, e.g., \cite[Proposition~2.3.1]{Ox}. A minimal edge cut of a graph is also called a \emph{bond}. The loops of $M(G)^*$ are exactly the bridges of $G$, and $e,f$ are parallel elements of $M(G)^*$ if and only if $\{e,f\}$ is a minimal edge cut. Hence, $M(G)^*$ is simple, or equivalently $M(G)$ is cosimple, if and only if any minimal edge cut of $G$ has at least three elements, and, in particular, any edge of $G$ is contained in a cycle of $G$. 

Throughout this paper, all matroids are assumed to be non-free matroids, and, in particular, they are non-empty matroids.

\section{Cycle polytopes of matroids}\label{cycle polytopes}
   
Let $M$ be a matroid. A \emph{cycle} of $M$ is defined to be a disjoint union of some of its circuits. We denote the set of all cycles of $M$ by $\mathrm{Cyc}(M)$. In particular, $\emptyset\in \mathrm{Cyc}(M)$. 

Attached to $M$ is a polytope $P_{\mathrm{Cyc}}(M)$ in $\mathbb{R}^{E(M)}$, called the \emph{cycle polytope} of $M$, which is defined as the convex hull of the characteristic vectors of cycles of $M$. Here a characteristic vector $\chi_C$ of a cycle $C$ of $M$ is a $0/1$-vector in $\ZZ^{E(M)}$ whose $e^{th}$ coordinate is $1$ if $e\in C$ and $0$ otherwise. 

If $M$ is a binary matroid, then $P_{\mathrm{Cyc}}(M)$ is exactly the convex hull of all $0/1$-vectors $x$ in $\RR^{E(M)}$ such that $A x\equiv 0$ (mod $2$) where $A$ is the representation matrix of $M$ over $\FF_2$. It follows from \cite[Theorem~4.1]{BGr} that 
\[
\dim P_{\mathrm{Cyc}}(M)=\text{the~number~of~coparallel~classes~of}~M.
\]

In the following theorem we determine certain faces of the cycle polytope of a matroid. First, recall that a \emph{face} $F$ of a polytope $P$ is a subset of $P$ which is the intersection of $P$ with a hyperplane $H$ such that $P$ is entirely contained in one of the two half-spaces defined by $H$. The hyperplane $H$ is then called a \emph{supporting hyperplane} of $P$. Moreover, a \emph{morphism} of polytopes $P$ and $Q$ is a map $\varphi\colon P\rightarrow Q$ which can be extended to an affine map 
$\tilde{\varphi}\colon \mathrm{aff}(P)\rightarrow \mathrm{aff}(Q)$. 
If the morphism $\varphi$ is an isomorphism, then the polytopes $P$ and $Q$ are said to be \emph{affinely isomorphic}.

\begin{Theorem}\label{face}
Let $M$ be a matroid and let $M'$ be a matroid obtained from $M$ by 
\begin{enumerate}
	\item [{\em(a)}] a deletion, or
	\item [{\em(b)}] a series contraction, or 
	\item [{\em(c)}] a coloop contraction.
\end{enumerate}
Then $P_{\mathrm{Cyc}}(M')$ is affinely isomorphic to a face of $P_{\mathrm{Cyc}}(M)$. 
In particular, if $M'$ is a series minor of $M$, then $P_{\mathrm{Cyc}}(M')$ is affinely isomorphic to a face of $P_{\mathrm{Cyc}}(M)$. 	
\end{Theorem}

\begin{proof}
Let $E=E(M)$, $e\in E$ and $E'=E-\{e\}$. Then $\RR^{E'}$ is naturally (isomorphic to) a subspace of $\RR^E$. 
In the following, we denote the $f^{th}$ coordinate of a vector $u$ in $\RR^E$ or $\RR^{E'}$ by ${u}_{f}$, for any $f\in E$.  

(a) First assume that $M'=M\setminus e$. Let $H$ be the hyperplane in $\RR^{E}$ defined by $x_{e}=0$, and let
$F:=P_{\mathrm{Cyc}}(M)\cap H$. Then clearly $F$ is a face of $P_{\mathrm{Cyc}}(M)$.

We claim that $P_{\mathrm{Cyc}}(M')$ is affinely isomorphic to $F$. Since zero is an element of $\mathrm{aff}(P_{\mathrm{Cyc}}(M'))$ corresponding to the empty cycle of $M'$, we have that  $\mathrm{aff}(P_{\mathrm{Cyc}}(M'))=
\mathrm{span}(P_{\mathrm{Cyc}}(M'))$ is a subspace of $\RR^E$. 

Next, define 
\begin{eqnarray*}
	\varphi\colon P_{\mathrm{Cyc}}(M') &\rightarrow & F
\end{eqnarray*}
such that for any $w\in P_{\mathrm{Cyc}}(M')$,
\begin{displaymath}
{\varphi(w)}_f= \left \{\begin {array}{ll}
w_f&\mathrm{if}~~~f\neq e,\\
0&\text{if}~~~f=e.
\end{array}\right.
\end{displaymath}
The map $\varphi$ is well-defined, because 
$\mathcal{C}(M')=\{C\subseteq E-\{e\}: C\in \mathcal{C}(M)\}$ and thus for a characteristic vector ${\chi}_{C'}\in \RR^{E'}$ of a cycle $C'$ of $M'$ one knows that $\phi(\chi_{C'})$ is a characteristic vector of a cycle of $M$ which lies in $H$. Clearly, $\varphi$ is the restriction of the affine/linear map 
\begin{eqnarray*}
	\tilde{\varphi}\colon \mathrm{aff}(P_{\mathrm{Cyc}}(M'))&\rightarrow & \mathrm{aff}(F)\subseteq \RR^{E}
\end{eqnarray*}
such that for any $w\in \mathrm{aff}(P_{\mathrm{Cyc}}(M'))$,
\begin{displaymath}
{\tilde{\varphi}(w)}_f= \left \{\begin {array}{ll}
w_f&\mathrm{if}~~~f\neq e,\\
0&\text{if}~~~f=e.
\end{array}\right.
\end{displaymath}
Hence, $\varphi$ is a morphism of the involved polytopes. By analogous reasons, 
\begin{eqnarray*}
	\psi\colon F&\rightarrow & P_{\mathrm{Cyc}}(M')
\end{eqnarray*}
defined as 
\[
{\psi(v)}_f=v_f \quad \text{for any} \quad v\in F \quad \text{and} \quad f\neq e,
\]
is well-defined and is a morphism of the polytopes $F$ and $P_{\mathrm{Cyc}}(M')$ which is the inverse to $\varphi$. This concludes the proof of~(a).   

(b) Next assume that $e$ and $f$ are coparallel for some $f\in E$, namely $\{e,f\}$ is a cocircuit of $M$, and let $M'=M/e$. 
We claim that $P_{\mathrm{Cyc}}(M')$ is affinely isomorphic to $P_{\mathrm{Cyc}}(M)$. 

Note that, since $\{e,f\}$ is a cocircuit of $M$, it follows from \eqref{intersection of circuit and cocircit} that any circuit of $M$ either contains both of $e$ and $f$, or contains none of them. As any cycle of $M$ is a disjoint union of circuits, the same property holds for any cycle of $M$. For $v\in P_{\mathrm{Cyc}}(M)$, this yields  
$v_{e}=v_{f}$. 

We define 
\begin{eqnarray*}
	\varphi \colon P_{\mathrm{Cyc}}(M')&\rightarrow & P_{\mathrm{Cyc}}(M)
\end{eqnarray*}
such that for any $w\in P_{\mathrm{Cyc}}(M')$,
\begin{displaymath}
{\varphi(w)}_{e'}= \left \{\begin {array}{ll}
w_f&\mathrm{if}~~~e'=e,\\
w_{e'}&\text{if}~~~e'\neq e.
\end{array}\right.
\end{displaymath}
It follows from the discussion above and the relation of cycles of $M'$ and $M$ that $\varphi$ is well-defined. Moreover, $\varphi$ is the restriction of the affine/linear map 
\begin{eqnarray*}
\tilde{\varphi}\colon \mathrm{aff}(P_{\mathrm{Cyc}}(M')) &\rightarrow&  \mathrm{aff}(P_{\mathrm{Cyc}}(M))
\end{eqnarray*}
defined as 
\begin{displaymath}
{\tilde{\varphi}(w)}_{e'}= \left \{\begin {array}{ll}
w_f&\mathrm{if}~~~e'=e,\\
w_{e'}&\text{if}~~~e'\neq e,
\end{array}\right.
\end{displaymath}
for any $w\in \mathrm{aff}(P_{\mathrm{Cyc}}(M'))$. Hence, $\varphi$ is a morphism of the involved polytopes. By analogous reasons,  
\begin{eqnarray*}
	\psi \colon P_{\mathrm{Cyc}}(M)&\rightarrow & P_{\mathrm{Cyc}}(M')
\end{eqnarray*}
defined as 
\[
{\psi(v)}_{e'}=v_{e'} \quad \text{for any} \quad v\in P_{\mathrm{Cyc}}(M) \quad \text{and} \quad e'\neq e,
\]
is a well-defined morphism of polytopes which is the inverse to $\varphi$. Hence, $P_{\mathrm{Cyc}}(M')$ is isomorphic to $P_{\mathrm{Cyc}}(M)$. 

(c) Finally, assume that $e$ is a coloop of $M$. Then, by \eqref{intersection of circuit and cocircit}, $e$ is not contained in any circuit of $M$. This implies that 
$\MC(M/e)=\MC(M)=\MC(M\setminus e)$, which yields  
\begin{equation}\label{cycles-coloop}
\mathrm{Cyc}(M/e)=\mathrm{Cyc}(M)=\mathrm{Cyc}(M\setminus e).
\end{equation}  
This implies that 
$P_{\mathrm{Cyc}}(M/e)=P_{\mathrm{Cyc}}(M\setminus e)\subseteq \RR^{E'}$. Thus, by part~(a), it follows that $P_{\mathrm{Cyc}}(M/e)$ is affinely isomorphic to a face of $P_{\mathrm{Cyc}}(M)$. Note that \eqref{cycles-coloop} also implies that the $e$-th coordinate of any vertex (and hence any element) of $P_{\mathrm{Cyc}}(M)$ is equal to zero. Hence, in this case, $P_{\mathrm{Cyc}}(M/e)$ is indeed affinely isomorphic to $P_{\mathrm{Cyc}}(M)$ itself, since by using the notation of part~(a), we have $F=P_{\mathrm{Cyc}}(M)$. 
\end{proof}

In the following we consider two important special cases of cycle polytopes arising from graphs.

\begin{Example}\label{cut polytope}
	{\em 
	Let $G=(V,E)$ be a graph, and let $M=M(G)$.       
	\begin{enumerate}
		\item \textbf{Eulerian subgraph polytopes}: It is a classical fact in graph theory that a graph $H$ is \emph{Eulerian}, (i.e.~all of its vertices have even degree) if and only if its edge set is the disjoint union of the edge sets of some cycles of $H$. It follows that the cycles of $M$ are exactly the edge sets of the Eulerian subgraphs of $G$. Then the cycle polytope of $M$ is indeed the convex hull of the incident vectors of the Eulerian subgraphs of $G$, namely the vectors $\delta_H\in \RR^{E}$ with  
		\begin{displaymath}
		\delta_{H,e}= \left \{\begin {array}{ll}
		1&\mathrm{if}~~~e\in E(H),\\
		0&\mathrm{otherwise},
		\end{array}\right.
		\end{displaymath}  
		where $H$ is an Eulerian subgraph of $G$ and $e\in E$. This polytope is also known as the \emph{Eulerian subgraph polytope}; see, e.g., \cite{BGr}. In the following, we denote this polytope by $\mathrm{Euler}(G)$. 
		\\
		\item \textbf{Cut polytopes}: 
		Given a subset $A$ of $V$, the \emph{cut set} $\mathrm{Cut}(A)$ of $G$ is a subset of $E$ consisting of those edges of $G$ which have exactly one endpoint in $A$. The \emph{cut polytope} of $G$, which is denoted by ${\mathrm{Cut}}^{\square}(G)$, is the convex hull of the \emph{cut vectors} $\delta_{A}\in \mathbb{R}^{E}$ of $G$, which are defined as
		\begin{displaymath}
		\delta_{A,e}= \left \{\begin {array}{ll}
		1&\mathrm{if}~~~e\in \mathrm{Cut}(A),\\
		0&\mathrm{otherwise},
		\end{array}\right.
		\end{displaymath}
		for any $A\subseteq V$ and $e\in E$. The cut polytope of $G$ has been intensively studied by many authors; see, e.g., \cite{CKNR, D, DD, DL2, O1, O2, RS, SS}. 
		
		It is clear that any cut set $\mathrm{Cut}(A)$ is an edge cut (in the sense of Section~\ref{matroids}) if $A\neq \emptyset, V$. The converse is not true in general, but one can see that any minimal edge cut is a cut set. Hence, minimal edge cuts and minimal cut sets of $G$ coincide. Then, it follows from \cite[Exercises~4.1.27~and~4.1.28]{W} that a subset $C$ of $E$ is a disjoint union of minimal edge cuts (i.e.~a cycle of $M^*$) if and only if $C=\mathrm{Cut}(A)$ for some $\emptyset\neq A\subset V$. Note that the zero vector corresponds to $A=\emptyset$ and the empty cycle. Therefore, in this case, the cycle polytope of $M^*$ is exactly the cut polytope of $G$.  
	\end{enumerate}
}
 \end{Example}

\section{Cycle algebra of a matroid and its algebra retracts}\label{algebra retracts}

Given a field $\KK$, associated to any lattice polytope $P\subseteq \RR^d$ is a toric algebra $\KK[P]$ whose generators bijectively correspond to the lattice points of $P$, namely, 
\[
\KK[P]=\KK[\mathbf{y^a}z:\mathbf{a}\in P\cap \ZZ^d]. 
\] 
If $P\subseteq \RR^d_{\geq 0}$, then the algebra $\KK[P]$ is a $\KK$-subalgebra of the polynomial ring $\KK[y_1,\ldots,y_d,z]$ where 
$\mathbf{y^a}=y_1^{a_1}\cdots y_d^{a_d}$ with $\mathbf{a}=(a_1,\ldots,a_d)\in \ZZ^d$. The toric algebra $\KK[P]$ is naturally a standard graded $\KK$-algebra (generated in degree~1) induced by setting $\deg (z)=1$ and $\deg (y_i)=0$ for all $i=1,\ldots ,d$.     

\medskip
Now, let $M$ be a matroid. We define the \emph{cycle algebra} of $M$ to be the toric algebra $\KK[P_{\mathrm{Cyc}}(M)]$, and to simplify the notation, we denote it by $\KK[\mathrm{Cyc}(M)]$. More precisely, 
\[
\KK[\mathrm{Cyc}(M)]:=\KK[\mathbf{y}^{C}z:C\in \text{Cyc}(M)]
\]
is a $\KK$-subalgebra of the polynomial ring 
$R_M=\KK[y_e, z: e\in E(M)]$, where the monomial $\mathbf{y}^C=\prod_{e\in C} y_e$ corresponds to the characteristic vector $\chi_C$ of $C$. Let $S_M=\KK[x_C: C\in \mathrm{Cyc}(M)]$. One gets a representation for this toric ring by the following homogeneous $\KK$-algebra homomorphism:
\begin{eqnarray*}
	\phi_{M}\colon S_M &\longrightarrow& \KK[\mathrm{Cyc}(M)]
	\\
	x_C &\mapsto& \mathbf{y}^{C}z,
\end{eqnarray*}
for any $C\in \mathrm{Cyc}(M)$. The defining ideal of this algebra is denoted by $I_{\mathrm{Cyc}(M)}$ and we call it the \emph{cycle ideal} of $M$. The cycle ideal is a graded ideal in $S_M$ generated by pure homogeneous binomials of the form~$\prod_{i=1}^dx_{C_i}-\prod_{i=1}^dx_{D_i}$ with $C_i, D_i\in \mathrm{Cyc}(M)$ and $d\geq 1$.  

\begin{Example}\label{cut ideal}
	{\em 
		Let $G$ be a graph.       
		\begin{enumerate}
			\item \textbf{Eulerian algebras}: Let $M=M(G)$. Then, we refer the associated cycle algebra and the cycle ideal, as the \emph{Eulerian algebra} and the \emph{Eulerian ideal} of $G$, respectively, which are denoted by $\KK[\mathrm{Euler}(G)]$ and $I_{\mathrm{Euler}(G)}$, respectively. To the best of our knowledge, the Eulerian algebra has not been studied in the literature before. 
			\\
			\item \textbf{Cut algebras}: Let $M=M(G)^*$. If $G$ is connected, then the corresponding cycle algebra and ideal is naturally identified with the \emph{cut algebra} and \emph{cut ideal} of $G$, introduced in \cite{SS}, (see also \cite{RS}). If $G$ is disconnected, then the cycle algebra of $M$ is naturally isomorphic to the cut algebra of $G$, but cycle ideals could differ by linear forms (see \cite[Proposition~3.2]{RS} and Lemma~\ref{mu}). We denote the cut algebra and ideal of $G$ by $\KK[\mathrm{Cut}(G)]$ and $I_{\mathrm{Cut}(G)}$, respectively. 
		\end{enumerate}
	}
\end{Example} 

The main goal of this section is to provide useful algebra retracts of cycle algebras of matroids generalizing corresponding results of \cite{RS}. First let us recall the well-known definition of an algebra retract of a graded algebra. Here, if it is not stated otherwise, by ``graded" we mean standard $\ZZ$-graded.

\begin{Definition}\label{retract-def}
	{\em Let $A$ and $B$ be graded $\KK$-algebras and let
		$\iota\colon A\rightarrow B$ be an (injective) homogeneous $\KK$-algebra homomorphism. Then $A$ is called an \emph{algebra retract} of $B$, if there exists a homogeneous (surjective) homomorphism of $\KK$-algebras
		$\gamma\colon B\rightarrow A$ such that $\gamma \circ \iota=\mathrm{id}_{A}$.	}
\end{Definition}

It is clear from the above definition that if $A$, $B$ and $C$ are graded $\KK$-algebras such that $A$ is an algebra retract of $B$, and $B$ is an algebra of $C$, then $A$ is an algebra retract of $C$. Note that in this paper the homogeneous homomorphisms are not necessarily of degree~$0$, and also we do not consider more general types of algebra retracts, namely, arbitrary (non-graded) ones.

The highest degree of an element of a minimal homogeneous generating set, the projective dimension~$\projdim_A(I)$ and the Castelnuovo-Mumford regularity~$\reg_A(I)$, and more generally, the graded Betti numbers~$\beta_{i,j}^A(I)$ of a defining ideal~$I$ of a graded $\KK$-algebra~$A/I$ where $A$ is a polynomial ring over a field $\KK$  do not increase by retraction. A precise statement is: 

\begin{Proposition}\label{Betti}
	{\em (}\cite[Corollary~2.5]{OHH}{\em )}
	Let $R=A/I$ and $S=B/J$ be graded $\KK$-algebras where $A$ and $B$ are polynomial rings over a field $\KK$. Suppose that $R$ is an algebra retract of $S$, and $I$ and $J$ are graded ideals containing no linear forms. Then
	\begin{enumerate}\label{betti-retract}
		\item [{\em (a)}] $\beta_{i,j}^A(I)\leq \beta_{i,j}^B(J)$ for all $i,j\in \ZZ_{\geq 0}$.
		\item [{\em (b)}] $\projdim_A(I)\leq \projdim_B(J)$.
		\item [{\em (c)}] $\reg_A(I)\leq \reg_B(J)$.
	\end{enumerate}
\end{Proposition}

Observe that in Section~\ref{highest degree} we will verify that cycle ideals of matroids contain no linear forms, and hence the above proposition is applicable to them.

\medskip
We divide this section into two parts in which we discuss different types of algebra retracts.
First, in Proposition~\ref{face retract}, we consider a class of face retracts of cycle algebras of matroids. The following is indeed a consequence of Theorem~\ref{face}. 

\begin{Proposition}\label{face retract}
	Let $M$ and $M'$ be two matroids such that $M'$ is obtained by a sequence of series minors and coloop contractions from $M$. Then $\KK[\mathrm{Cyc}(M')]$ is an algebra retract of $\KK[\mathrm{Cyc}(M)]$.		
\end{Proposition}

\begin{proof}
	In general, for a lattice polytope
	$P\subseteq \RR^d_{\geq 0}$
	and a face $F$ of $P$, one has that
	$\KK[F]$ is an algebra retract of $\KK[P]$
	(see, e.g., \cite[Corollary 4.34]{BG}). This fact and applying Theorem~\ref{face} conclude the proof.
\end{proof}

\begin{Remark}\label{contraction-face}
{\em 	Note that in spite of deletions, contractions (and hence arbitrary minors), in general, do not necessarily provide a face. For example, if $M$ is the cographic matroid of the $4$-cycle $C_4$ and $e$ is an edge of $C_4$,  then $P_{\mathrm{Cyc}}(M/e)\iso P_{\mathrm{Cyc}}(M(P_4)^*)\iso \mathrm{Cut}^{\square}(P_4)$ is not affinely isomorphic to any faces of $P_{\mathrm{Cyc}}(M)$. Indeed, this example even shows that contraction may not provide an algebra retract at all; see \cite[Remark~4.4]{RS}. }
\end{Remark}

Next, we provide other algebra retracts which are not necessarily induced by faces of cycle polytopes. The following definition is motivated by the proof of \cite[Theorem~5.4]{RS}:

\begin{Definition}\label{matroidal retract-def}
	Let $M$ and $M'$ be two matroids with $E(M')\subseteq E(M)$. Suppose that there exist two maps $\lambda \colon \mathrm{Cyc}(M')\rightarrow \mathrm{Cyc}(M)$ and $\pi \colon \mathrm{Cyc}(M)\rightarrow \mathrm{Cyc}(M')$ which satisfy the following conditions: 
	\begin{enumerate}
		\item[{\em(a)}] $\pi \circ \lambda=\mathrm{id}_{\mathrm{Cyc}(M')}$. 
		\item[{\em(b)}] If $\sum_{i=1}^d \chi_{C'_i}=\sum_{i=1}^d \chi_{D'_i}$ for some $C'_i,D'_i\in \mathrm{Cyc}(M')$ with $i=1,\ldots,d$ and $d\geq 1$, then 
		$\sum_{i=1}^d \chi_{\lambda(C'_i)}=\sum_{i=1}^d \chi_{\lambda(D'_i)}$.	
		\item[{\em(c)}] If $\sum_{i=1}^d \chi_{C_i}=\sum_{i=1}^d \chi_{D_i}$ for some $C_i,D_i\in \mathrm{Cyc}(M)$ with $i=1,\ldots,d$ and $d\geq 1$, then 
		$\sum_{i=1}^d \chi_{\pi(C_i)}=\sum_{i=1}^d \chi_{\pi(D_i)}$.
	\end{enumerate} 
	Then we say that $M'$ is a \textbf{matroidal retract} of $M$. 
\end{Definition}

\begin{Remark}\label{condition (iii)}
{\em	Using the notation in Definition~\ref{matroidal retract-def}, we would like to observe that a natural example of a map $\pi$ which satisfies condition~(c), is:
	\begin{equation}
      \pi(C)=C\cap E(M'), \quad \text{for any}~~C\in \mathrm{Cyc}(M).
	\end{equation} 
}
\end{Remark}

The following theorem allows us to build up certain algebra retracts for the cycle algebra of a matroid which arise from matroidal retracts.  

\begin{Theorem}\label{matroidal retract-theorem}
	Let $M$ be a matroid and let $M'$ be a matroidal retract of $M$.  
	Then $\KK[\mathrm{Cyc}(M')]$ is an algebra retract of $\KK[\mathrm{Cyc}(M)]$. 
\end{Theorem} 

\begin{proof}
Let $\lambda$ and $\pi$ be two maps with the desired properties of Definition~\ref{matroidal retract-def}. First we define the homomorphism  
\begin{eqnarray*}
	\tilde{\lambda}\colon S_{M'}&\longrightarrow& S_M
	\\
	x_C&\mapsto& x_{\lambda(C)}.
\end{eqnarray*}
Let $\iota$ be the induced map from $\tilde{\lambda}$, defined as follows:
\begin{eqnarray*}
	\iota \colon S_{M'}/I_{\mathrm{Cyc}(M')}&\longrightarrow& S_M/I_{\mathrm{Cyc}(M)}
	\\
	f+I_{\mathrm{Cyc}(M')}&\mapsto& \tilde{\lambda}(f)+I_{\mathrm{Cyc}(M)},
\end{eqnarray*}
for any $f\in S_{M'}$. To check that $\iota$ is a well-defined map, it is enough to show that $\tilde{\lambda}(I_{\mathrm{Cyc}(M')})\subseteq I_{\mathrm{Cyc}(M)}$. Let $f=\prod_{i=1}^d x_{C'_i}-\prod_{i=1}^d x_{D'_i}$ be an element of a minimal homogeneous generating set of $I_{\mathrm{Cyc}(M')}$ for some $d\geq 1$, where $C'_i$'s and $D'_i$'s are cycles of $M'$. Then  $\phi_{M'}(f)=0$, since $I_{\mathrm{Cyc}(M')}=\ker \phi_{M'}$. This together with the definition of the map $\phi_{M'}$ implies that $\prod_{i=1}^d{\mathbf{y}}^{C'_i}=\prod_{i=1}^d\mathbf{y}^{D'_i}$, or equivalently  
\begin{equation}\label{well-defined1}
\sum_{i=1}^d \chi_{C'_i}=\sum_{i=1}^d \chi_{D'_i}.
\end{equation}
Since $M'$ is a matroidal retract of $M$, Definition~\ref{matroidal retract-def}~(b) and (\ref{well-defined1}), yield 
$\sum_{i=1}^d \chi_{\lambda(C'_i)}=\sum_{i=1}^d \chi_{\lambda(D'_i)}$, 
and then  
\[
\prod_{i=1}^d{\mathbf{y}}^{\lambda(C'_i)}=\prod_{i=1}^d\mathbf{y}^{\lambda(D'_i)}.
\]
The latter equality implies that $\phi_M(\tilde{\lambda}(f))=0$. Hence $\tilde{\lambda}(f)\in I_{\mathrm{Cyc}(M)}$, as desired.

Next, we define the homomorphism  
\begin{eqnarray*}
	\tilde{\pi}\colon S_{M}&\longrightarrow& S_{M'}
	\\
	x_C&\mapsto& x_{\pi(C)}.
\end{eqnarray*}
Let $\gamma$ be the induced map from $\tilde{\pi}$, as follows:
\begin{eqnarray*}
	\gamma \colon S_{M}/I_{\mathrm{Cyc}(M)}&\longrightarrow& S_{M'}/I_{\mathrm{Cyc}(M')}
	\\
	f+I_{\mathrm{Cyc}(M)}&\mapsto& \tilde{\pi}(f)+I_{\mathrm{Cyc}(M')},
\end{eqnarray*}
for any $f\in S_{M}$. To see that $\gamma$ is a well-defined map, it suffices to show that $\tilde{\pi}(I_{\mathrm{Cyc}(M)})\subseteq I_{\mathrm{Cyc}(M')}$. Let $f=\prod_{i=1}^d x_{C_i}-\prod_{i=1}^d x_{D_i}$ be an element of a minimal homogeneous generating set of $I_{\mathrm{Cyc}(M)}$ for some $d\geq 1$, where $C_i$'s and $D_i$'s are cycles of $M$. Then  $\phi_{M}(f)=0$, since $I_{\mathrm{Cyc}(M)}=\ker \phi_{M}$. Thus, $\prod_{i=1}^d{\mathbf{y}}^{C_i}=\prod_{i=1}^d\mathbf{y}^{D_i}$, or equivalently 
\begin{equation}\label{well-defined2}
\sum_{i=1}^d \chi_{C_i}=\sum_{i=1}^d \chi_{D_i}.
\end{equation}
Since $M'$ is a matroidal retract of $M$, it follows from Definition~\ref{matroidal retract-def}~(c) and (\ref{well-defined2}) that   
$\sum_{i=1}^d \chi_{\pi(C_i)}=\sum_{i=1}^d \chi_{\pi(D_i)}$,
and then  
\[
\prod_{i=1}^d{\mathbf{y}}^{\pi(C_i)}=\prod_{i=1}^d\mathbf{y}^{\pi(D_i)}.
\]
Therefore, $\phi_{M'}(\tilde{\pi}(f))=0$, and hence $\tilde{\pi}(f)\in I_{\mathrm{Cyc}(M')}$, as desired.

Finally, we need to show that $\gamma \circ \iota=\mathrm{id}_{S_{M'}/I_{\mathrm{Cyc}(M')}}$. Indeed, for any $f\in S_{M'}$, we have
\[
\gamma \circ \iota(f+I_{\mathrm{Cyc}(M')})=\gamma(\tilde{\lambda} (f)+I_{\mathrm{Cyc}(M)})=\tilde{\pi} (\tilde{\lambda}(f))+I_{\mathrm{Cyc}(M')}=f+I_{\mathrm{Cyc}(M')},
\]   
where the last equality follows from Definition~\ref{matroidal retract-def}~(a), since $M'$ is a matroidal retract of $M$. Hence 
$\KK[\mathrm{Cyc}(M')]$ is an algebra retract of $\KK[\mathrm{Cyc}(M)]$.  
\end{proof} 

In Proposition~\ref{face retract}, deletions and more generally series minors and coloop contractions were discussed to obtain algebra retracts. In the following we consider another type of contractions in binary matroids which are important for us in the sequel.

\begin{Definition}\label{binary matroidal retract-def}
Let $M$ be a binary matroid, and let $E$ and $E'$ be two disjoint subsets of $E(M)$ such that $|E|=|E'|=s$ with $E=\{e_1,\ldots,e_s\}$ and $E'=\{e'_1,\ldots,e'_s\}$ for some $s\geq 1$. Suppose that the following conditions hold:
\begin{enumerate}
	\item[{\em(a)}] $E'\in \mathcal{C}(M)$.
	\item[{\em(b)}] For any $C\in \mathcal{C}(M)$ and $p=0,\ldots,s-1$, one has $C\cap E'=\{e'_{i_1},\ldots,e'_{i_p}\}$ if and only if  either 
	\begin{enumerate}
		\item[{\em(i)}] $C\cap E=\{e_{i_1},\ldots,e_{i_p}\}$, or
		\item[{\em(ii)}] $C\cap E=E-\{e_{i_1},\ldots,e_{i_p}\}$.
	\end{enumerate}
\end{enumerate}
Then we say that the binary matroid $M/E'$ is a \textbf{binary matroidal retract} of $M$. 
\end{Definition}

\begin{Remark}\label{symmetic}
{\em	Observe that condition~(b) in Definition~\ref{binary matroidal retract-def} is symmetric in terms of $E$ and $E'$. Indeed, it is easily seen that condition~(b) is equivalent to the following condition: 
	
	For any $C\in \mathcal{C}(M)$ and $p=0,\ldots,s-1$, one has:
	\begin{enumerate}
		\item[(i)] $C\cap E'=\{e'_{i_1},\ldots,e'_{i_p}\}$ implies that 
		\[
		C\cap E=\{e_{i_1},\ldots,e_{i_p}\}~ \text{or}~E-\{e_{i_1},\ldots,e_{i_p}\},
		\] 
		\item[(ii)] $C\cap E=\{e_{i_1},\ldots,e_{i_p}\}$ implies that 
		\[
		C\cap E'=\{e'_{i_1},\ldots,e'_{i_p}\}~\text{or}~E'-\{e'_{i_1},\ldots,e'_{i_p}\}.
		\]
	\end{enumerate}
}
\end{Remark}

Next, we show that binary matroidal retracts result in algebra retracts in the case of binary matroids. For this purpose, we use the following theorem which determines a nice property of binary matroids.

\begin{Theorem}\label{cycles in binary matroids}
	{\em (}\cite[Corollary~9.3.7]{Ox}{\em )}
	Let $M$ be a binary matroid, let $C\in \mathcal{C}(M)$ and let $e\in E(M)-C$. Then, either $C\in \mathcal{C}(M/e)$ or $C$ is a disjoint union of two circuits of $M/e$. In both cases, $M/e$ has no other circuits contained in $C$.  
\end{Theorem} 

The next fact is also useful in the proof of Theorem~\ref{binary matroidal retract-theorem} below.  

\begin{Remark}\label{circuit of cardinality at least two}
{\em Note that it is clear that if $C$ is a circuit of an arbitrary matroid $M$ with $|C|\geq 2$ and $e\in C$, then $C-\{e\}$ is a circuit of $M/e$, see also \cite[Page~317]{Ox}. }	
\end{Remark}

\begin{Theorem}\label{binary matroidal retract-theorem}
 A binary matroidal retract of a binary matroid is a matroidal retract.  
\end{Theorem}

\begin{proof}
	Let $M$ be a binary matroid and let $M'$ be a binary matroidal retract of $M$. Then there are disjoint subsets $E$ and $E'$ of $E(M)$ as in Definition~\ref{binary matroidal retract-def} such that $M'=M/E'$.	The goal is to show that maps $\lambda$ and $\pi$ exist with the desired properties of Definition~\ref{matroidal retract-def}. At first, we define the map $\lambda$ as follows:
	\begin{eqnarray*}
		\lambda\colon \mathrm{Cyc}(M/E')&\longrightarrow& \mathrm{Cyc}(M)
		\\
		C'&\mapsto& C'\cup \{e'_{i_j}:e_{i_j}\in C'\}.
	\end{eqnarray*}
    Note that in particular, $\lambda(\emptyset)=\emptyset$. 
    We have to verify that $\lambda$ is well-defined, and for this, it remains to prove that for any non-empty $C'\in \mathrm{Cyc}(M/E')$, one has $\lambda(C')\in \mathrm{Cyc}(M)$. So, let $C'\neq \emptyset$ be a cycle of $M/E'$. Then there exist pairwise disjoint circuits $C_1,\ldots,C_t$ of $M/E'$, for some $t\geq 1$, such that  $C'=\cup_{i=1}^t C_i$. By definition of $\lambda$, it is immediately clear that $\lambda(C_i)$'s are pairwise disjoint as well, and $\lambda(C')=\cup_{i=1}^t\lambda(C_i)$. To see that $\lambda(C')\in \mathrm{Cyc}(M)$, it is enough to show that for each $i=1,\ldots,t$, we have $\lambda(C_i)\in \mathrm{Cyc}(M)$. For simplicity, we may assume right away that $C'\in \mathcal{C}(M/E')$. This implies that $C'=C-E'$ for some $C\in \mathcal{C}(M)$ with $C\neq E'$. Let 
    \[
    C'\cap E=\{e_{j_1},\ldots,e_{j_q}\}
    \] 
    for some $q$ with $0\leq q\leq s$. Then $C\cap E=\{e_{j_1},\ldots,e_{j_q}\}$, since $C'\cap E=C\cap E$. We distinguish three cases: 
    
    \emph{Case}~1. Assume that $q=0$. Then Remark~\ref{symmetic}~(ii)  yields $C\cap E'=\emptyset$, since $C\cap E=C'\cap E=\emptyset$ and since $C$ and $E'$ are two distinct circuits. Hence, $C=C'$. Thus,
    $\lambda(C')=C'=C\in \mathrm{Cyc}(M)$.
    
    \emph{Case}~2. Assume that $q=s$. This implies that $E\subseteq C'$, and hence $\lambda(C')=C'\cup E'$. It follows directly from Definition~\ref{binary matroidal retract-def} that $E'\in \mathcal{C}(M)$ and $C'\cap E'=\emptyset$ by the choice of $C'=C-E'$. By a similar argument as in Case~1, in Definition~\ref{binary matroidal retract-def}~(b) only the case $C\cap E'=\emptyset$ is possible, because $E=C'\cap E=C\cap E$. Thus, $C'=C$ and $\lambda(C')\in \mathrm{Cyc}(M)$.
    
    
    \emph{Case}~3. Assume that $1\leq q\leq s-1$. Then $\lambda(C')=C'\cup \{e'_{j_1},\ldots,e'_{j_q}\}$. By Definition~\ref{binary matroidal retract-def}, we have either 
    $C\cap E'=\{e'_{j_1},\ldots,e'_{j_q}\}$ or 
    $C\cap E'=E'-\{e'_{j_1},\ldots,e'_{j_q}\}$. In the first case we get $\lambda(C')=C$ which is a circuit of $M$, while in the second case one obtains $\lambda(C')=C\Delta E'$ which is a cycle of $M$, by Theorem~\ref{binary}, since $M$ is a binary matroid and $E'\in \mathcal{C}(M)$.         
    
    All together we see that $\lambda$ is indeed well-defined. Related to $\lambda$ it remains to show that condition~(b) in Definition~\ref{matroidal retract-def} holds. Let $C_i, D_i\in \mathrm{Cyc}(M/E')$ for $i=1,\ldots,d$ with $d\geq 1$, and assume that 
    \begin{equation}\label{union of cycles}
     \sum_{i=1}^d \chi_{C_i}=\sum_{i=1}^d \chi_{D_i}. 
    \end{equation}
    The goal is to prove that  
    $\sum_{i=1}^d \chi_{\lambda(C_i)}=\sum_{i=1}^d \chi_{\lambda(D_i)}$. 
    For any $e\in E(M)$, let $m_e$ and $m'_e$ denote the coordinate corresponding to $e$ in $\sum_{i=1}^d \chi_{\lambda(C_i)}$ and $\sum_{i=1}^d \chi_{\lambda(D_i)}$, respectively. Using this notation, it remains to see that $m_e=m'_e$.  Note that if $e\in E(M/E')$, then it is clear by the definition of $\lambda$ that $m_e$ is equal to the $e$-th coordinate in $\sum_{i=1}^d \chi_{C_i}$, and $m'_e$ is equal to the $e$-coordinate of $\sum_{i=1}^d \chi_{D_i}$. Thus, in this case, $m_e=m'_e$, according to (\ref{union of cycles}). Next consider the remaining case $e=e'_j$ in $E'$ for some $j\in \{1,\ldots,s\}$. 

    At first assume that $m_e=0$. So, $e=e'_j\notin \lambda(C_i)$ for all $i=1,\ldots,d$, and then $e_j\notin C_i$ for all $1=1,\ldots,d$. Thus, by (\ref{union of cycles}), we deduce that $e_j\notin D_i$ for all $i$, and hence $e=e'_j\notin \lambda(D_i)$ for all $i=1\ldots,d$. Therefore, it follows that also $m'_{e}=0$ as desired. 

    In the second case assume that $m_e=t$ for some positive integer $t$. Then there exist exactly $t$ different indices   $k_1,\ldots,k_t\in\{1,\ldots,d\}$ such that $e\in \lambda(C_{k_{\ell}})$ for all $\ell=1,\ldots,t$, because the sets $\lambda(C_i)$ are disjoint as observed above. The definition of $\lambda$ and the assumption $e=e'_j$ yield that $C_{k_1},\ldots,C_{k_t}$ are the only cycles among $C_i$'s which contain $e_j$. Thus, the $e_j$-th coordinate of 
    $\sum_{i=1}^d \chi_{C_i}$ is equal to $t$, and hence by (\ref{union of cycles}), the same coordinate of $\sum_{i=1}^d \chi_{D_i}$ equals to $t$. This means that there are exactly $t$ different indices $h_1,\ldots, h_t$ for which $D_{h_1},\ldots,D_{h_t}$ contain $e_j$, and hence $\lambda(D_{h_1}),\ldots,\lambda(D_{h_t})$ are the only ones among $\lambda(D_i)$'s which contain $e=e'_j$. This implies that $m'_e=t$. This concludes the verification of Definition~\ref{matroidal retract-def}~(b).
     
     Continuing the discussion that Definition~\ref{matroidal retract-def} holds, we have to define an appropriate map $\pi$. For this we set:
     \begin{eqnarray*}
     	\pi\colon \mathrm{Cyc}(M)&\longrightarrow& \mathrm{Cyc}(M/E')
     	\\
     	C&\mapsto& C-E'.
     \end{eqnarray*}
     Observe that in particular, $\pi(\emptyset)=\pi(E')=\emptyset$. 
     As a first task we have to see that $\pi$ is well-defined and for this one has to show that for any non-empty cycle $C$ of $M$, $C-E'$ is a cycle of $M/E'$. Note that the cycle $C$ can be written as $C=\cup_{i=1}^tC_i$ for some $t\geq 1$, where $C_i$'s are pairwise disjoint circuits of $M$ for $i=1,\ldots,t$. Hence, $C-E'$ is the disjoint union of the sets $C_1-E',\ldots,C_t-E'$. Thus, it is enough to show that for each $i=1,\ldots,t$, one has  $C_i-E'\in \mathrm{Cyc}(M/E')$. So, without loss of generality, we may assume that $C\in \mathcal{C}(M)$. To prove that $C-E'\in \mathrm{Cyc}(M/E')$, we distinguish two case:
     
     \emph{Case}~1. Assume that $C\cap E'=\emptyset$. Then $C-E'=C$. Since $M$ is a binary matroid, Theorem~\ref{cycles in binary matroids} implies that either $C\in \mathcal{C}(M/e'_1)$ or $C=C_1\cup C_2$ where $C_1,C_2\in \mathcal{C}(M/e'_1)$ and $C_1\cap C_2=\emptyset$. Since $M/e'_1$ is a binary matroid as well, one can apply Theorem~\ref{cycles in binary matroids} to this matroid for contraction with respect to the element $e'_2$,  and each of the probable circuits $C$ or $C_1$ and $C_2$. Then, by repeating this procedure, after a finite number of steps, it follows that $C$ is either a circuit of $M/E'$ or a disjoint union of certain circuits of $M/E'$. Hence $C$ is a cycle of $M/E'$, as desired.   
     
     \emph{Case}~2. Assume that $C\cap E'\neq \emptyset$. If $C=E'$, then the claim is trivially true. So, suppose that $C\neq E'$. Then, $|C|\geq 2$, since $E'$ and $C$ are both circuits and can not contain each other. Moreover, if 
     $C\cap E'=\{e'_{j_1},\ldots,e'_{j_{\ell}}\}$ for some $1\leq \ell <s$, then for any proper subset $T$ of $\{e'_{j_1},\ldots,e'_{j_{\ell}}\}$, one has $|C-T|\geq 2$. Thus, according to Remark~\ref{circuit of cardinality at least two}, we have 
     $C-\{e'_{j_1},\ldots,e'_{j_{\ell}}\}\in \mathcal{C}(M/\{e'_{j_1},\ldots,e'_{j_{\ell}}\})$. Since $M/\{e'_{j_1},\ldots,e'_{j_{\ell}}\}$ is a binary matroid, and $E'\cap (C-\{e'_{j_1},\ldots,e'_{j_{\ell}}\})=\emptyset$, similar to the Case~1, by applying Theorem~\ref{cycles in binary matroids} repeatedly, it follows that $C-E'$ is indeed a cycle of $M/E'$.     
     
     It is clear that for any $C\in \mathrm{Cyc}(M)$, we have  
     $\pi(C)=C\cap E(M/E')$, which by Remark~\ref{condition (iii)} implies that Definition~\ref{matroidal retract-def}~(c) also holds. 
     
     Finally, according to the definitions of the maps $\lambda$ and $\pi$, it is obvious that $\pi \circ \lambda$ is the identity map on the $\mathrm{Cyc}(M/E')$, which verifies condition~(a) in the definition of a matroidal retract. This concludes the proof that $M'=M/E'$ is a matroidal retract of $M$. 
\end{proof}

\begin{Problem}\label{binary matroidal retract-problem}
	The proof of Theorem~\ref{binary matroidal retract-theorem} uses at several places the binary assumption. We leave it as an interesting question to decide whether an analogous statement of Theorem~\ref{binary matroidal retract-theorem} holds or not, if one drops the word "binary" everywhere in Definition~\ref{binary matroidal retract-def}. 
\end{Problem}

Motivated by the main results of this section, we introduce a type of minors arising from deletions and certain contractions:

\begin{Definition}\label{g-series minor-definition}
   Let $M$ and $M'$ be matroids where $M'$ is obtained from $M$ by a sequence of 
   \begin{enumerate}
   	\item [{\em(a)}] deletions, 
   	\item [{\em(b)}] series contractions,
   	\item [{\em(c)}] coloop contractions, and
   	\item [{\em (d)}] binary matroidal retracts.
   \end{enumerate}
Then we call $M'$ a \textbf{generalized series minor} {\em (}or shortly, \textbf{g-series minor}{\em )} of $M$. 
\end{Definition}

\begin{Remark}\label{comparison of minor types}
{\em 
	Observe that for a $g$-series minor, one is allowed to apply two more operations ``coloop contraction" and ``binary matroidal retract" on matroids compared to the case of a series minor. In particular, every series minor is a $g$-series minor. }
\end{Remark}

The next corollary is an immediate consequence of Proposition~\ref{face retract}, Theorem~\ref{matroidal retract-theorem} and Theorem~\ref{binary matroidal retract-theorem}. 

\begin{Corollary}\label{algebra retract-corollary} 
 Let $M$ be a binary matroid and $M'$ be a $g$-series minor of $M$. Then $\KK[\mathrm{Cyc}(M')]$ is an algebra retract of $\KK[\mathrm{Cyc}(M)]$.  
\end{Corollary}

\section{Generalized series minors of cographic matroids}\label{Cographic case}

In this section, we consider cographic matroids, which are in particular binary, and investigate certain $g$-series minors of them. The main goal here is to verify that Corollary~\ref{algebra retract-corollary} recovers one of the main results from~\cite{RS} (see \cite[Theorem~5.4]{RS}). 

First, let us recall some definitions from \cite{RS}. Let $G=(V,E)$ be a simple graph. Recall that an \emph{induced subgraph} $G_T$ on any non-empty subset $T$ of $V$ is the subgraph of $G$ whose vertex set is $T$ and whose edges are those edges of $G$ which have both endpoints in $T$. 

Next, let $v\in V$. Then
\[
N_G(v)=\{w\in V: w~\mathrm{is~a~neighbor~of}~v~\mathrm{in}~G\},
\]
where a vertex $w\in V$ is said to be a \emph{neighbor} of $v$ in $G$, if it is adjacent to $v$. Furthermore,  
\[
N_G[v]=N_G(v)\cup \{v\} ~~ \mathrm{and}~~ N_G(T)=\cup_{v\in T} N_{G}(v),
\]
for every non-empty subset $T$ of $V$.

Assume that $V=W\cup W'$ with $W\cap W'=\emptyset$ and $W,W'\neq \emptyset$, and let $H=G_{W}$ be the induced subgraph of $G$ on $W$.
Suppose that there exists a vertex $v\in W$ with 
$W\cap N_{G}(W')\subseteq N_{H}[v]$. Then $H$ is said to be a \emph{neighborhood-minor} of $G$. 

Recall that a minor of $G$ is a graph obtained from $G$ by applying a sequence of the operations ``edge deletion" and ``edge contraction", together with disregarding the isolated vertices.  
One can check that a neighborhood minor of a graph is indeed a minor of it, but the converse does not hold in general. For example, by removing an edge from a graph $G$, one obtains a minor, but the obtained graph is not an induced subgraph of $G$ while neighborhood minors are always induced subgraphs by definition. 

We would like to remark that the notion of neighborhood-minors was defined in \cite[Definition~5.2]{RS}. But, as it was mentioned in \cite[Remark~5.3]{RS}, in the special case that $|W'|=1$ it has been previously considered in studying cut polytopes for different purposes; see, e.g., \cite[Theorem~2]{D}. 

\medskip
In the following theorem, we see the relationship between neighborhood-minors of graphs and $g$-series minors of cographic matroids.

\begin{Theorem}\label{neighborhood-g-series minor-corollary}
 Let $G$ be a simple graph and $H$ be a neighborhood-minor of $G$. Then $M(H)^*$ is a g-series minor of $M(G)^*$. 
\end{Theorem}

\begin{proof}
 Assume that $H$ is a neighborhood-minor of $G$ with $H=G_W$, where $V(G)=W\cup W'$ with $W\cap W'=\emptyset$ and $W,W'\neq \emptyset$
 such that $W\cap N_G(W')\subseteq N_H[v]$ for some $v\in W$.
 
 If $W\cap N_{G}(W')=\emptyset$, then $G$ is just the disjoint union of the two graphs $H$ and $G_{W'}$, and equivalently $M(G)^*$ is the direct sum of $M(H)^*$ and $M(G_{W'})^*$. In this case, 
 it is easily seen that $M(H)^*$ is a $g$-series minor of $M(G)^*$, 
 as desired, e.g.~by using a suitable number of deletions. 
 
 Next we consider the case $W\cap N_{G}(W')\neq \emptyset$. First suppose that $G_{W'}$ is connected. Let $(W\cap N_{G}(W'))-\{v\}=\{v_1,\ldots,v_s\}$ for some $s\geq 0$, and let $e_i=\{v,v_i\}$ be the corresponding edges for all $i=1,\ldots,s$. 
 
 Now, we consider at first certain actions at the level of graphs and then interpret them in the level of corresponding cographic matroids. Consider the following steps:    
  \begin{enumerate}
  	\item[(i)] By consecutive contractions of edges as well as deleting all loops which might occur during contractions, we get from $G_{W'}$ just a vertex of $W'$, say $w'$, since $G_{W'}$ is connected. In this way a new graph $G'$ is obtained from $G$ on the vertex set $W\cup \{w'\}$ such that $G'_W=G_W$, and  
  	\begin{equation}\label{neighborhood}
  	W\cap N_{G'}(w')=W\cap N_{G}(W').
  	\end{equation}
  	In particular, $W\cap N_{G'}(w')\subseteq N_H[v]$.  
  	
  	\item[(ii)] We distinguish two cases:
  	\begin{enumerate}
  		\item [(1)] Suppose that $v\in N_{G'}(w')$. In this case we contract the edge $e=\{v,w'\}$. Note that some parallel edges to $e_1,\ldots,e_s$ might appear. By removing those new edges, we get $H$ as a minor of $G$. 
  		\item [(2)] Suppose that $v\notin N_{G'}(w')$. Thus, it follows that $s\geq 1$. By (\ref{neighborhood}), we deduce that $w'$ is adjacent to all $v_1,\ldots,v_s$. Let $e'_i=\{w',v_i\}$ for $i=1,\ldots,s$. Then, by removing the edges $e'_1,\ldots, e'_s$, we obtain a graph, say $H'$, which is just the union of $H$ and the isolated vertex $w'$. Then, we disregard the isolated vertex and so we get the desired result. 
  	\end{enumerate}
  \end{enumerate}
Using the aforementioned steps, we interpret the equivalent and analogous steps at the level of cographic matroids:

\begin{enumerate}
	\item[(i)$'$] As we discussed in Section~\ref{matroids}, the contractions of edges and deletions of loops in step~(i), result in certain deletions and coloop contractions, respectively, in the corresponding cographic matroids, and hence it follows that $M(G')^*$ is a $g$-series minor of $M(G)^*$.     
	\item[(ii)$'$]  Corresponding to the cases~(1) and~(2) we have:  
	\begin{enumerate}
		\item [(1)$'$] Case~(1) implies that $e$ is an element of $M(G')^*$ which is deleted from $M(G')^*$. The possible occurring parallel edges in~(1), imply certain coparallel elements in the corresponding cographic matroid, and we contract them to obtain $M(H)^*$. Thus, in this case, $M(H)^*$ is a $g$-series minor of $M(G')^*$, and hence a $g$-series minor of $M(G)^*$ by~(i)$'$. 
		\item [(2)$'$] Case~(2) implies that $e$ is not an element of $M(G')^*$, but $e_1,\ldots,e_s$ are elements of $M(G')^*$ for some $s\geq 1$. Then one sees that 
		\[
		E(M(G')^*)=E(H)\cup \{e'_1,\ldots,e'_s\}.
		\]
		Deleting the edges $e'_1,\ldots,e'_s$ from $G'$, yields the contraction $M(H')^*=M(G')^*/\{e'_1,\ldots,e'_s\}$ from $M(G')^*$. As $w'$ is just an isolated vertex of $H'$, one knows that $M(H')^*$ and $M(H)^*$ are isomorphic matroids. We claim that $M(G')^*/\{e'_1,\ldots,e'_s\}$ is a binary  matroidal retract of $M(G')^*$. 
		Let $E=\{e_1,\ldots,e_s\}$ and $E'=\{e'_1,\ldots,e'_s\}$, which are disjoint subsets of $E(M(G')^*)$. Then it is clear that $E'$ is a minimal edge cut of $G'$, because deleting the edges in $E'$ from $G'$, disconnects the vertex $w'$ from $H$, and none of the proper subsets of $E'$ disconnects $G'$. Thus, $E'$ is a circuit of $M(G')^*$, and hence Definition~\ref{binary matroidal retract-def}~(a) is fulfilled. To verify~(b) in the same definition, let $C\in \mathcal{C}(M(G')^*)$. Then according to 
		Example~\ref{cut polytope}, it follows that $C=\mathrm{Cut}(A)$ is a minimal cut set of $G'$ for a set $\emptyset\neq A\subseteq V(G)$. Suppose that $C\cap E'=\{e'_{i_1},\ldots,e'_{i_p}\}$ for some $p\in \{0,\ldots,s-1\}$. If either $v,w'\in A$ or $v,w'\in A^c$, then just by the definition of a cut set we deduce that either  $A^c\cap \{v_{1},\ldots,v_{s}\}=\{v_{i_1},\ldots,v_{i_p}\}$ or $A\cap \{v_{1},\ldots,v_{s}\}=\{v_{i_1},\ldots,v_{i_p}\}$, respectively, which implies that 
		$C\cap E=\{e_{i_1},\ldots,e_{i_p}\}$. 
		If $v\in A$ and $w'\in A^c$ or vice versa, then similarly it follows that 
		$C\cap E=E-\{e_{i_1},\ldots,e_{i_p}\}$. Hence we obtain one implication of the statement of condition~(2) in Definition~\ref{binary matroidal retract-def}. The other implication follows similarly by symmetry.     
	\end{enumerate}  
\end{enumerate}

Finally, suppose that $G_{W'}$ is disconnected with connected components on disjoint sets of vertices $W'_1,\ldots,W'_r$. Set $G_0=G$ and $G_t=G_{V-\cup_{i=1}^t W'_i}$ for all $t=1,\ldots,r$. Since $H$ is a neighborhood-minor of $G$, 
it follows that $G_t$ is a neighborhood-minor of $G_{t-1}$ for all $t=1,\ldots,r$. Notice that $G_{W'_t}$ is clearly connected for each $t$. Then, by changing the role of $G_{W'}$ by $G_{W'_t}$ and the role of $G_W$ by $G_t$ in the previous cases of the proof, no matter 
$(V-\cup_{i=1}^t W'_i)\cap N_G(W'_t)$ is empty or not, 
it follows that $M(G_t)^*$ is a $g$-series minor of $M(G_{t-1})^*$ for all $t$. Hence, $M(H)^*$ is a $g$-series minor of $M(G)^*$ which concludes the proof.  
\end{proof}

The next corollary follows immediately from Corollary~\ref{algebra retract-corollary} and Theorem~\ref{neighborhood-g-series minor-corollary}.

\begin{Corollary}\label{neighborhood-retract-corollary}
	{\em (}\cite[Theorem~5.4]{RS}{\em)} Let $G$ be a simple graph and $H$ be a neighborhood-minor of $G$. Then $\KK[\mathrm{Cut}(H)]$ is an algebra retract of $\KK[\mathrm{Cut}(G)]$.  
\end{Corollary}

We see that neighborhood-minors yield a method for cographic matroids to obtain $g$-series minors and thus algebra retracts on the level of algebras. We end this section by posing the following problem concerning the Eulerian algebras of graphs to find similar results in this case.

\begin{Problem}\label{Eulerian retract}
Let $H$ and $G$ be two graphs. It would be interesting to provide some explicit graphical conditions on $H$ and $G$ under which  $M(H)$ is a g-series minor of $M(G)$, and hence $\KK[\mathrm{Euler}(H)]$ is an algebra retract of $\KK[\mathrm{Euler}(G)]$.  
\end{Problem}

\section{The comparison of the highest degrees of minimal homogeneous generators of cycle ideals of matroids}\label{highest degree}

In this section, we consider cycle ideals of matroids in more detail and in particular discuss certain situations where one can relate or compare the highest degree of minimal homogeneous generating sets of cycle ideals of two matroids with each other. Let denote by $\mu(M)$ the highest degree of an element of a minimal homogeneous generating set of $I_{\mathrm{Cyc}(M)}$ and set $\deg (0)=-\infty$. 

First we characterize when a cycle ideal is zero. Let $M$ be matroid. For simplicity, let 
\[
d(M):=\text{the number of coparallel classes of}~M.
\]
As it was mentioned in Section~\ref{cycle polytopes}, it is known that 
$\dim P_{\mathrm{Cyc}}(M)=d(M)$. This together with \cite[Proposition~4.22]{BG} imply that the Krull dimension of the cycle algebra of $M$ is given by
\[
\dim \KK[\mathrm{Cyc}(M)]= d(M)+1,
\]
and hence 
\begin{equation}\label{height}
\height I_{\mathrm{Cyc}(M)}=|\mathrm{Cyc}(M)|-d(M)-1. 
\end{equation}

\begin{Lemma}\label{zero ideal}
 Let $M$ be a matroid. Then the following statements are equivalent:
 	\begin{enumerate}
 	\item[{\em(a)}] $I_{\mathrm{Cyc}(M)}=\langle 0 \rangle$;
 	\item[{\em(b)}] $d(M)=|\mathrm{Cyc}(M)|-1$.
 \end{enumerate}
\end{Lemma} 

\begin{proof}
	The desired result follows from \eqref{height} and since clearly  $I_{\mathrm{Cyc}(M)}=\langle 0 \rangle$ if and only if 
	$\height I_{\mathrm{Cyc}(M)}=0$, since it is a prime ideal.  
\end{proof}

\begin{Corollary}\label{zero ideal-cosimple}
 Let $M$ be a cosimple matroid. Then the following statements are equivalent: 
 \begin{enumerate}
 	\item[{\em(a)}] $I_{\mathrm{Cyc}(M)}=\langle 0 \rangle$;
 	\item[{\em(b)}] $|E(M)|=|\mathrm{Cyc}(M)|-1$.
 \end{enumerate}
\end{Corollary}

\begin{proof}
Since $M$ is cosimple, $\{e\}$ is a coparallel class of $M$ for any $e\in E(M)$ and these are the only coparallel classes, as it was mentioned in Section~\ref{matroids}. Thus, $d(M)=|E(M)|$ which together with Lemma~\ref{zero ideal} yield the desired result.  
\end{proof}

The following example provides a cosimple matroid satisfying the equivalent conditions of Corollary~\ref{zero ideal-cosimple}.

\begin{Example}\label{F7-1}
{\em The \emph{Fano matroid} $F_7$ is a matroid with the ground set $E=\{1,\ldots,7\}$ whose bases are all 3-subsets of $E$, except the ones shown in Figure~\ref{Fano} by straight lines and a curve, i.e.~$C_1=\{1,2,6\}$, $C_2=\{1,3,5\}$, $C_3=\{2,3,4\}$, $C_4=\{2,5,7\}$, $C_5=\{3,6,7\}$, $C_6=\{1,4,7\}$, $C_7=\{4,5,6\}$. This implies that 
	\[
	\mathcal{C}(F_7)=\{C_i,C_i^c : i=1,\ldots,7\},
	\]	
	where $C_i^c$ denotes the complementary set of $C_i$ for each~$i$. Then, $F_7$ is in particular a simple matroid. Thus, one can observe that $F_7^*$ is a cosimple matroid with 
	\[
	\mathcal{C}(F_7^*)=\{C_i^c : i=1,\ldots,7\},
	\]  	
and hence $\mathrm{Cyc}(F_7^*)=\{\emptyset\}\cup \mathcal{C}(F_7^*)$. Thus, Corollary~\ref{zero ideal-cosimple} implies that $I_{\mathrm{Cyc}(F_7^*)}=\langle 0 \rangle$.  	
}
\end{Example}

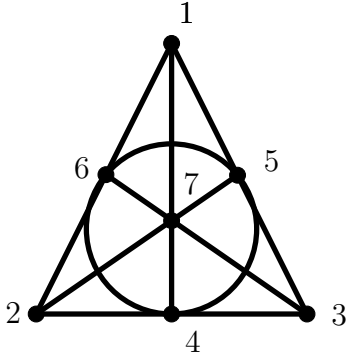
\begin{figure}[h!]
\centering 
\begin{tikzpicture}[scale = 1.2]
\definecolor{zzttqq}{rgb}{1,1,1}
\fill[line width=2pt,color=zzttqq] (-1,5) -- (-2.5,2) -- (0.5,2) -- cycle;
\draw [line width=2pt] (-1,5)-- (-2.5,2);
\draw [line width=2pt] (-2.5,2)-- (0.5,2);
\draw [line width=2pt] (0.5,2)-- (-1,5);
\draw [line width=2pt] (-1.0009636766678007,2.9433265697243844) circle (0.9433270619571432cm);
\draw [line width=2pt] (-1,5)-- (-1,2);
\draw [line width=2pt] (-2.5,2)-- (-0.2695,3.539);
\draw [line width=2pt] (-1.7275,3.545)-- (0.5,2);
\draw (-1.04,5.57) node[anchor=north west] {1};
\draw (-1.04,5.57) node[anchor=north west] {1};
\draw (-2.96,2.25) node[anchor=north west] {2};
\draw (0.65,2.23) node[anchor=north west] {3};
\draw (-0.97,1.93) node[anchor=north west] {4};
\draw (-0.1,3.93) node[anchor=north west] {5};
\draw (-2.2,3.85) node[anchor=north west] {6};
\draw (-0.99,3.68) node[anchor=north west] {7};
\begin{scriptsize}
\draw [fill=black] (-1,5) circle (2.5pt);
\draw [fill=black] (-2.5,2) circle (2.5pt);
\draw [fill=black] (0.5,2) circle (2.5pt);
\draw [fill=black] (-1.7275,3.545) circle (2.5pt);
\draw [fill=black] (-0.2695,3.539) circle (2.5pt);
\draw [fill=black] (-1,2) circle (2.5pt);
\draw [fill=black] (-1,3.0349697377269673) circle (2.5pt);
\end{scriptsize}
\end{tikzpicture}
\caption{Geometric representation of the Fano matroid $F_7$}
\label{Fano}
\end{figure}

In the special case of cographic matroids of simple graphs, the cycle ideal is rarely zero. Indeed, $I_{\mathrm{Cut}(G)}$ is zero if and only if $G$ is the complete graph $K_2$ or $K_3$, see \cite[Proposition~3.1]{RS}. But, by Lemma~\ref{zero ideal} and Example~\ref{F7-1}, it seems that in general there are more interesting cases with zero cycle ideals. Now, it is natural to pose the following problem:

\begin{Problem}\label{zero characterization}
It would be interesting to give an excluded minor characterization or explicit list of all matroids $M$ with $I_{\mathrm{Cyc}(M)}=\langle 0 \rangle$. 
\end{Problem}

Next, we investigate the numbers $\mu(M)$ for matroids whose cycle ideals are not zero.

\begin{Lemma}\label{mu}
	Let $M$ be a matroid. Then the following statements hold: 
	\begin{enumerate}
		\item [{\em(a)}] 
		${(I_{\mathrm{Cyc}(M)})}_1=\langle 0 \rangle$. In particular, if $I_{\mathrm{Cyc}(M)}\neq \langle 0 \rangle$, then  $\mu(M)\geq 2$.  
		
	    \item [{\em(b)}] If $M'$ is a matroid such that $\KK[\mathrm{Cyc}(M')]$ is an algebra retract of $\KK[\mathrm{Cyc}(M)]$, then $\mu(M')\leq \mu(M)$.   
	\end{enumerate}
\end{Lemma}

\begin{proof}
	Part (a) follows, since by definition
	no binomial of the form $x_C-x_D$ belongs to $I_{\mathrm{Cyc}(M)}$, where $C\neq D$ are cycles of $M$. Part~(b) follows from part~(a) together with Proposition~\ref{Betti}~(a).   
\end{proof}

We see by Lemma~\ref{mu}~(a) that cycle ideals never contain linear forms. As a consequence of Lemma~\ref{mu}~(b) together with Proposition~\ref{face retract}, Theorem~\ref{matroidal retract-theorem} and Corollary~\ref{algebra retract-corollary}, some operations are considered in the next corollary under which the highest degree of a minimal homogeneous set of generators of the corresponding ideals does not increase.  

\begin{Corollary}\label{mu comparison}
	Let $M$ be a matroid and let $M'$ be a minor of $M$. Assume that one of the following holds:  
	\begin{enumerate}
		\item[{\em(a)}] $M'$ is a series minor of $M$, 
		\item[{\em(b)}] $M'$ is a binary matroidal retract of $M$ and $M$ is binary, or
		\item[{\em(c)}] $M'$ is a g-series minor of $M$ and $M$ is binary.   
	\end{enumerate} 
	Then 
	\[
	\mu(M')\leq \mu(M).
	\] 
\end{Corollary}

Next, the following theorem provides some situations where the highest degrees of minimal homogeneous sets of generators of cycle ideals are the same, though the underlying matroids are not. 

\begin{Theorem}\label{simplification}
	Let $M$ and $M'$ be two matroids where $M'$ is obtained by either 
	\begin{enumerate}
		\item [{\em(a)}] a coloop contraction, or
		\item [{\em(b)}] a series contraction
	\end{enumerate}  
	of $M$. Then $\mu(M)=\mu(M')$.   	
\end{Theorem}

\begin{proof}
	\begin{enumerate}
		
	   \item [(a)] Assume that $M'=M/e$ where $e$ is a coloop of $M$. In the proof of Theorem~\ref{face} part~(c), we observed that, in this case, $\mathrm{Cyc}(M)=\mathrm{Cyc}(M')$ and, in particular, no cycle of $M$ contains $e$. This, by definition, implies that 
	$I_{\mathrm{Cyc}(M)}=I_{\mathrm{Cyc}(M')}$ (as ideals in $S_M=S_{M'}$), and hence $\mu(M)=\mu(M')$. 
	
	  \item [(b)] Next assume that $M'=M/e$ where $\{e,f\}$ is a cocircuit of $M$ for some $f\in E(M)$. By Corollary~\ref{mu comparison}~(a), we have $\mu(M')\leq \mu(M)$. We claim that the other inequality also holds. For this, at first observe that 
	\[
	\mathrm{Cyc}(M')=\{D-\{e\}: D\in \mathrm{Cyc}(M)\}.
	\] 
	This indeed follows, since we have $\mathcal{C}(M')=\{C-\{e\}: C\in \mathcal{C}(M)\}$ which means that the sets $C-\{e\}$ are all minimal. The latter equation follows from the fact that if  $C\in \mathcal{C}(M)$ with $e\in C$, then $C-\{e\}$ is not contained in any circuits of $M$ except $C$, since as it was discussed in the proof of Theorem~\ref{face}~(b), any circuit of $M$ contains $e$ if and only if it contains $f$. Indeed, in the critical case which is the case that $e\in C$, if $C-\{e\}\subseteq D$ for some $D\in \mathcal{C}(M)$, then $f\in D$ which yields $e\in D$. Therefore, $C\subseteq D$, and hence $C=D$, since both are circuits. 
	
	Define the homomorphism $\alpha:S_{M'}\rightarrow S_M$ with     
	\begin{displaymath}
	\alpha(x_C)= \left \{\begin {array}{ll}
	x_C&\mathrm{if}~~~f\notin C,\\
	x_{C\cup \{e\}}&\mathrm{if}~~~f\in C,
	\end{array}\right.
	\end{displaymath}
	for any $C\in \mathrm{Cyc}(M')$. Then $\alpha$ clearly provides an isomorphism between $S_{M'}$ and $S_M$. Let $\{f_1,\ldots,f_k\}$ be a minimal homogeneous generating set for $I_{\mathrm{Cyc}(M')}$. We claim that $\{\alpha(f_1),\ldots,\alpha(f_k)\}$ is a generating set for $I_{\mathrm{Cyc}(M)}$. Then, it follows that $\mu(M)\leq \mu(M')$, since $\alpha$ is a homogeneous homomorphism of degree zero. 
	
	First we need to show that $\alpha(f_j)\in I_{\mathrm{Cyc}(M)}$ for each $j$. Since $f_j\in I_{\mathrm{Cyc}(M')}$, we may assume that $f_j$ is a homogeneous binomial, namely $f_j=\prod_{i=1}^dx_{C_i-\{e\}}-\prod_{i=1}^dx_{D_i-\{e\}}$ for some $d$ and $C_i,D_i\in \mathrm{Cyc}(M)$. It follows that $\phi_{M'}(f_j)=0$, and hence $\sum_{i=1}^{d}\chi_{C_i-\{e\}}=\sum_{i=1}^{d}\chi_{D_i-\{e\}}$. This implies that the number of those $C_i$'s and $D_i$'s which contain $f$ are the same. Since exactly such $C_i$'s and $D_i$'s also contain $e$, it follows that  
	$\sum_{i=1}^{d}\chi_{C_i}=\sum_{i=1}^{d}\chi_{D_i}$ which yields 
	\[\alpha(f_j)=
	\prod_{i=1}^dx_{C_i}-\prod_{i=1}^dx_{D_i}\in I_{\mathrm{Cyc}(M)},
	\] 
	as desired. Next, assume that $\mathcal{G}$ is a generating set of homogeneous binomials for $I_{\mathrm{Cyc}(M)}$, and let $g\in \mathcal{G}$ with $g=\prod_{i=1}^dx_{C_i}-\prod_{i=1}^dx_{D_i}$ for some $d$ and some cycles $C_i$ and $D_i$ of $M$. Then 
	$\alpha^{-1}(g)=\prod_{i=1}^dx_{C_i-\{e\}}-\prod_{i=1}^dx_{D_i-\{e\}}$ which is an element of $I_{\mathrm{Cyc}(M')}$ by a similar argument as above. Hence, we have $\alpha^{-1}(g)=\sum_{i=1}^{k}h_if_i$ for some $h_i\in S_{M'}$,  which implies that $g=\sum_{i=1}^{k} \alpha(h_i) \alpha(f_i)$. Thus, $\{\alpha(f_1),\ldots,\alpha(f_k)\}$ is indeed a generating set for $I_{\mathrm{Cyc}(M)}$. This concludes the proof.  
\end{enumerate}          	   
\end{proof}

\section{Cycle ideals generated in small degrees}\label{degree 2}

In this section, cycle ideals of matroids $M$ with small values for $\mu(M)$ are considered. In particular, we discuss when cycle ideals are generated by quadrics. In the following denoted by~$K_n$ we mean the complete graph with~$n$ vertices.     

\begin{Lemma}\label{degree2}
Let $M$ be a binary matroid. Consider the following statements:
\begin{enumerate}
\item[{\em(a)}] $\mu (M)\leq 2$.
\item[{\em(a$'$)}] $M$ is $M(K_4)$-minor free.
\item[{\em(b)}] $M$ is $M(K_4)$-g-series minor free.
\item[{\em(c)}] $M$ has no $M(K_4)$ as a minor obtained by deletions, series contractions or coloop contractions.  
\end{enumerate}
Then the following implications hold:
\[
(a)\implies (b)\implies (c),
\]
and 
\[
(a')\implies (b)\implies (c).
\]
\end{Lemma}  

\begin{proof}
Note that since $M(K_4)$ is self-dual, its cycle polytope is affinely isomorphic to the cut polytope of $K_4$. Therefore,~(a)~implies~(b), by Corollary~\ref{mu comparison}, since $I_{\mathrm{Cyc}(M(K_4))}$ has a minimal generator of degree $4$, see, e.~g.,~\cite[Example~7.1]{RS}. The other implications simply follow from Definition~\ref{g-series minor-definition}. 
\end{proof}

\begin{Remark}\label{F7-2}
{\em Here, we discuss the statements of Lemma~\ref{degree2} in the cases of some well known matroids. 
\begin{enumerate}
	\item If $M$ is a connected graphic or cographic matroid, then by \cite[Corollary~5.4.12]{Ox} and the fact that $M(K_4)$ is a self-dual matroid it follows that~(c) implies~(a$'$). So~(b),~(c) and~(a$'$) are equivalent in these cases. But, the next part shows that these equivalences do not hold in general.    
	
	\item As we observed in Example~\ref{F7-1}, the cycle ideal of $F_7^*$ is zero. It follows from \cite[Proposition~6.4.8]{Ox} that $F_7$ and hence its dual $F_7^*$ are binary matroids. Therefore, $F_7^*$  satisfies conditions~(a),~(b) and~(c) in Lemma~\ref{degree2}. On the other hand, all contractions $F_7^*/e$, for any $e\in E(F_7^*)$, are isomorphic to $M(K_4)$, see, e.g., \cite[Example~1.5.6]{Ox}. Thus, $F_7^*$ does not satisfy~(a$'$).  
\end{enumerate}
}
\end{Remark} 

By a \emph{series-parallel network} we mean a 2-connected graph obtained from the complete graph~$K_2$ by subdividing and duplicating edges. It is clear that any series-parallel network is a planar graph. There are several ways to describe this class of graphs, see, e.g., \cite{Ep}. There are also some equivalent statements in terms of the graphic matroid, see, e.g., \cite[Corollary~5.4.12]{Ox}. Using the two latter descriptions, Engstr\"om showed in \cite{En} the following which in particular proved \cite[Conjecture~3.5]{SS}.

\begin{Proposition}\label{Eng}
	{\em (See} \cite[Corollary~2.8]{En}{\em )} 
	Let $G$ be a series-parallel network. Then $\mu (M(G)^*)\leq 2$.  
\end{Proposition}  

In the next theorem we continue the discussion of  Remark~\ref{F7-2}~(1) for graphic and cographic matroids. 
In the cographic case, this generalizes the main result of \cite{En} related to cut ideals, while the graphic case related to Eulerian ideals is new to the best of our knowledge.  

\begin{Theorem}\label{degree2-characterization}
	Let $M$ be a graphic or cographic matroid of a simple connected graph. Then the statements~(a), (b), (c) and (a$'$) in Lemma~\ref{degree2} are equivalent.
\end{Theorem}  

\begin{proof}
According to Lemma~\ref{degree2} and Remark~\ref{F7-2}~(a), it remains to show that the equivalent conditions~(b), (c) and (a$'$) imply~(a) in the graphic and cographic cases.  
	
First assume that $M=M(G)^*$ where $G$ is a simple connected graph, and assume that (a$'$) holds. Then, $M^*=M(G)$ is also connected and $M(K_4)$-minor free, by self-duality of $M(K_4)$. Hence, by \cite[Corollary~5.4.12]{Ox}, it follows that $G$ is a series-parallel network. Thus, Theorem~\ref{Eng} yields $\mu (M(G)^*)\leq 2$, and hence~(a) holds.     	  

Next, assume that $M=M(G)$ where $G$ is a simple connected graph, and assume that (a$'$) holds. Then, \cite[Corollary~5.4.12]{Ox} implies that $G$ is a series-parallel network, and hence a planar graph. So, it follows from \cite[Corollary~6.6.6]{Ox} that $M$ is a cographic matroid. Therefore, the desired result follows from the previous case.   
\end{proof}   

\begin{Remark}\label{preconj}
	{\em
  \begin{enumerate}
  	\item A graphic matroid $M(G)$ is a connected simple binary matroid for which the statements~(a), (b), (c) and (a$'$) in Lemma~\ref{degree2} are equivalent according to Theorem~\ref{degree2-characterization}, when $G$ is a connected simple graph. 
  	 
  	\item A cographic matroid $M(G)^*$ is a connected cosimple binary matroid for which the statements~(a), (b), (c) and (a$'$) in Lemma~\ref{degree2} are equivalent according to Theorem~\ref{degree2-characterization}, when $G$ is a connected simple graph. 
  	
  	\item The matroid $F_7^*$ is an example of a simple and cosimple connected binary matroid that is not cographic and for which the statements~(a), (b) and (c) in Lemma~\ref{degree2} are equivalent, but (a$'$) does not hold, as we saw in Remark~\ref{F7-2}. 
  \end{enumerate}	
}
\end{Remark}

Having Theorem~\ref{degree2-characterization} and Remark~\ref{preconj} in mind, we now pose the following conjecture:

\begin{Conjecture}\label{degree2-Conj}
Let $M$ be a connected binary matroid which is simple or cosimple. Then the statements~(a), (b) and (c) in Lemma~\ref{degree2} are equivalent.
\end{Conjecture} 

The next lemma also gives partial information about cycle ideals generated in degrees~at most~5.  

\begin{Lemma}\label{higher degrees}
	Let $M$ be a binary matroid. Consider the following statements:
	\begin{enumerate}
		\item[{\em(a)}] $\mu (M)\leq 5$;
		\item[{\em(b)}] $M$ is $M(K_5)^*$-g-series minor free;
		\item[{\em(c)}] $M$ has no $M(K_5)^*$ as a minor obtained by deletions, series contractions or coloop contractions. 
	\end{enumerate}
	Then the following implications hold:
	\[
	(a)\implies (b)\implies (c).
	\]
\end{Lemma}  

\begin{proof}
It follows from \cite[Table~1]{SS} that $\mu(M(K_5)^*)=6$. Therefore,~(a)~implies~(b), by Corollary~\ref{mu comparison}. The other implication is clearly obtained from Definition~\ref{g-series minor-definition}.  	 
\end{proof}

It would be interesting to provide certain classes of binary matroids for which the statements~(a),~(b) and~(c) in Lemma~\ref{higher degrees} are equivalent, or to give any explicit characterization of matroids $M$ with $\mu(M)\leq 5$. Here, we pose the following conjecture:  

\begin{Conjecture}\label{conj-degree5}
 Let $M$ be a connected binary matroid which is simple or cosimple. Then the statements~(a), (b) and (c) in Lemma~\ref{higher degrees} are equivalent.	
\end{Conjecture} 

Note that in the case of cographic matroids of a simple connected graph (i.e., cut ideals case) Conjecture~\ref{conj-degree5} was essentially stated in \cite[Conjecture~3.6]{SS}. In the latter paper, the authors even suggest $\mu (M(G)^*)\leq 4$ instead of $\mu (M(G)^*)\leq 5$. Also, note that \cite[Conjecture~3.6]{SS} was stated in terms of any minors in graphs, not special types of minors. Indeed, it is easily seen that if a simple graph $G$ has a complete graph~$K_n$ as a (graphical) minor, then it can be obtained only by contraction of edges, deletion of loops or deletion of multiple edges (as well as removal of isolated vertices). This means that, the corresponding cographic matroid $M(G)^*$ has $M(K_n)^*$ as a minor which is obtained only by deletions, series contractions or coloop contractions, (i.e. the three conditions mentioned in Lemma~\ref{higher degrees}~(c)). 

Seeing Theorem~\ref{degree2-characterization} as well as the two conjectures posed in this section, it is natural to ask if there are examples of g-series minors which are not obtained only from deletions, series contractions or coloop contractions. We end this section with such an example:

\begin{Example}\label{g-series not series}
	{\em 
	Let $M$ be the Fano plane $F_7$. Using the notation of Definition~\ref{binary matroidal retract-def}, let $E=\{e_1=4,e_2=5,e_3=3\}$ and $E'=\{e'_1=1,e'_2=2,e'_3=6\}$. Recall the circuits of $F_7$ from Example~\ref{F7-1}. We have that $E'=C_1$ is a circuit of $M$. It is easily seen that the desired conditions of Definition~\ref{binary matroidal retract-def} are satisfied. The circuits of $M/E'$ are as follows: 
	\[
	\{3,5\}, \{4,7\}, \{3,4\}, \{5,7\}, \{1,4,5\}, \{1,3,7\}.
	\] 
	Thus, $M/E'$ is a binary matroidal retract of $M$, and hence a g-series minor. One observes that this minor can not be obtained only by deletions, series contractions or coloop contractions. Indeed, $M\setminus E'$ is different from $M/E'$, since its only circuit is $\{3,4,5,7\}$. On the other hand, it can be seen that by replacing any of the three deletions in $M\setminus 1\setminus 2\setminus 6$ by contractions, we do not have series contractions or coloop contractions.    
}
\end{Example}

\end{document}